\documentclass[aap,press]{apt} \volno{60} \partno{}

\renewcommand{\crnotice}{}   %
\firstpage{1}
\usepackage[english]{babel}
\usepackage{color}
\usepackage{latexsym,bm,mathscinet}
\usepackage{epsfig}
\usepackage{array}
\usepackage{algorithm}
\usepackage{algpseudocode}
\usepackage[colorlinks,linkcolor=blue, citecolor=blue,urlcolor=blue]{hyperref}
\usepackage[dvipsnames]{xcolor}
\usepackage{booktabs}
\usepackage{tikz}
\usetikzlibrary{arrows.meta, calc}
\newcommand{\dd}{{\,\mathrm d}}
\newcommand{\si}{\sigma}

\newcommand{\ka}{\kappa}

\renewcommand{\phi}{\varphi}

\renewcommand{\scr}[1]{{\mathcal #1}}

\newcommand{\EE}{\mathbb{E}}

\newcommand{\PP}{\mathbb{P}}
\newcommand{\ind}{\mathbf{1}}
\newcommand{\RR}{\mathbb{R}}

\newcommand{\QQ}{\mathbb{Q}}

\newcommand{\Bm}{\begin{bmatrix}}
\newcommand{\Em}{\end{bmatrix}}

\newcommand{\T}{{\top}}
\DeclareMathOperator{\atan}{atan2}

\newcommand{\rad}{\mathrm{\shortparallel}}
\newcommand{\ang}{\mathrm{\perp}}

\newcommand{\simiid}{\stackrel{\rm iid}{\sim}}

{\theorembodyfont{\rmfamily} \newtheorem{ass}{\noindent Assumption}[section]}
{\theorembodyfont{\rmfamily} }
\paperno{????-00000}
\authornames{F.H. VAN DER MEULEN}
\shorttitle{Conditioning diffusions on a terminal bearing}
\renewcommand{\ams}[2]
           {\begin{center}
            \begin{minipage}{315.83pt}
            \small
            \noindent 2020 Mathematics Subject Classification:~Primary {\uppercase{#1}}\\
            \phantom{2020 Mathematics Subject Classification:~}Secondary {\uppercase{#2}}
            \end{minipage}
            \end{center}\par
            \newcommand\mscpri{#1}
            \newcommand\mscsec{#2}
            \normalsize
           }
\newcommand{\citet}{\cite}
\newcommand{\citep}{\cite}
\begin{document}

\title{Conditioning diffusions on a terminal bearing}

\authorone[Vrije Universiteit Amsterdam]{Frank van der Meulen}
\addressone{Department of Mathematics,
De Boelelaan 1111, 1081 HV Amsterdam, Netherlands}
\emailone{f.h.van.der.meulen@vu.nl}

\begin{abstract}
We study sampling of a diffusion on $[0,T]$ conditioned on its terminal position lying on a prescribed ray, i.e.\ on an exactly observed bearing. For a linear drift we derive closed-form expressions for the Doob $h$-function and its gradient. For a broad class of tracking models admitting a Kronecker block structure, the guiding term is obtained from a single row of a matrix exponential, and the conditioned process decomposes into independent radial and angular components. For a nonlinear drift, we construct a guided process using the $h$-function of the linear model and prove absolute continuity of the law of the true conditioned process with respect to the law of the guided process, with an explicit likelihood ratio. This result is used to devise an auxiliary particle filter for bearings observed sequentially in time without noise. Numerical experiments show that the proposed filter approximates the filtering distribution and estimates the log-likelihood considerably more accurately than the shifted Rayleigh filter, at the cost of a higher computing time.
\end{abstract}
\ams{60J60}{62M20; 65C05}

\section{Introduction}

Filtering for tracking a moving object has received much attention within the engineering and signal processing literature. Often the movement is described by a discrete-time state-space model. Here, we rather assume that  the underlying latent process --capturing position, velocity  and possibly acceleration as well-- is modelled by a continuous-time process.  Observations are nevertheless only obtained discretely in time. More specifically, if $Z_t =(X_t, V_t,\ldots)$ with $X_t, V_t \in \RR^p$, we assume the latent continuous-time process $Z\equiv (Z_t,\, t\ge 0)$ satisfies a stochastic differential equation (SDE) of the form
\begin{equation}\label{eq:Z}
	\dd Z_t = B Z_t \dd t  + q f(t,Z_t)\dd t + q \dd W_t, \qquad Z_0 =z_0,
\end{equation}
with $B\in\RR^{d\times d}$, $q\in \RR^{d\times m}$, $f\colon [0,\infty)\times\RR^d \to \RR^m$ and $W_t$ a standard Wiener process in dimension $m$. Commonly $m<d$, so that $Q=qq^\T$ is singular.  At times $t_1<t_2<
\cdots<t_n$ we observe $U_n=X_{t_n}/\|X_{t_n}\|$, implying that $X_{t_i}$ lies on the ray $\{Ru\colon R>0\}$ for $1\le i \le n$. 
 This setup includes  the nearly-constant velocity  and  Singer model in $\RR^2$ as also  the 3-dimensional constant-turn model for example. For an overview of common tracking models we refer to the survey in \cite{Li2003survey1}. Note that the density of $U$ conditional on $Z_0$ is generally not known in closed form.  In this paper we focus on sampling the process $Z$ on a fixed interval $[0,T]$ conditional on $Z_0=z_0$ and a realisation $u$ from $U$.  This is subsequently applied to resolve the filtering problem where at time $t_i$ we observe $u_{i+1}$ and wish to reconstruct $(Z_t,\, t\in [t_i, t_{i+1}])$. For this, we devise an auxiliary particle filter. The latter is challenging as we assume bearings to be observed without (extrinsic) noise. 

\subsection{Approach}

Suppose first $f\equiv 0$. We will refer to the corresponding model as the {\it linear model}, ``linear'' referring to the corresponding linear drift. 
  The essential tool for conditioning the process $Z$ is Doob's $h$-transform, a  general technique for conditioning stochastic processes.
The key result that is being used is the following: if $u\mapsto g_u(t,z)$ is the density of $U \mid Z_t=z$, then the conditioned process satisfies the SDE
\begin{equation}\label{eq:Zcirc}
	 \dd Z_t^\star  = B Z_t^\star \dd t  + Q r(t,Z_t^\star)\dd t + q\dd W_t^\star,
\end{equation}
where $Q=qq^\T$, $r(t,z)=\nabla_z \log g_u(t,z)$ and $W_t^\star$ is a  standard Wiener process in dimension $m$. Notationally, we often suppress the subscript $u$ from $g$.  Since $X_T$ given $Z_t=z$ is Gaussian, $g(t,z)$ is the density of a projected normal distribution. At a single observation time this density is classical: it underlies the exact one-step bearing update of the shifted Rayleigh filter \cite{Clark2007SRF}. Our focus is on the conditioned process on path space. We derive a closed form for the drift of the conditioned process  (Theorem \ref{thm:general case g and nabla log g}), which simplifies considerably when the conditional covariance of $X_T$ is isotropic. For a block structure shared by most tracking models, we prove a decomposition of the conditioned process into a radial and an angular process driven by independent noise (Theorem \ref{thm:closed-system}). If the initial state $z_0$ is unknown, we show how samples from the posterior distribution can be obtained under a Gaussian prior (Theorem \ref{thm:prior}).

 In case $f\neq 0$, following the approach by \cite{schauer2017guided}, we define the guided process $Z^\circ$ to be the process solving the SDE
\begin{equation}\label{eq:guided-nonlinear}
	 \dd Z_t^\circ  = B Z_t^\circ \dd t + q f(t,Z^\circ_t) \dd t  + Q r(t,Z_t^\circ)\dd t + q\dd W_t^\circ,
\end{equation}
with $W_t^\circ$ a standard Wiener process in dimension $m$. We prove that $Z_T^\circ$ is on the ray 
and prove that the distribution (on path space) of the  true conditioned process is absolutely continuous with respect to the distribution of the guided process with  explicit likelihood ratio (Theorem \ref{prop:ac-nonlinear}). This result is subsequently used to devise an auxiliary particle filter. 

Numerical experiments illustrate our results. Our main example is a planar model for position and velocity where the drift of the velocity component is given by $v\mapsto -\ka\left(\|v\|^2-v_*^2\right)v$. For this case we devise an implicit discretisation scheme using splitting to sample the guided process in a numerically stable manner that ensures paths end up on the ray.

\subsection{Related work}

The discrete-time linear bearings-only tracking model is classical with many papers appearing in the engineering literature, see for instance   %
\cite{aidala1979kalman}, \cite{wang2009target},  \cite{Farina1999} and \cite{Mallick2022}. 
The model was used to illustrate the bootstrap filter in the seminal paper by \cite{gordon1993novel}. It has reappeared in the computational statistics community as an example in papers such as \cite{Fearnhead2002} and textbooks on Monte Carlo methods (\cite{Chopin2020introduction}, section 2.4.1 and \cite{robert1999monte}, section 14.3). 

For a Gaussian prior on the position, the exact update by a single bearing is the basis of the shifted Rayleigh filter (\cite{Clark2007SRF}, \cite{Ozelci2013}), which subsequently approximates the posterior by moment matching. The density $g$ is that of the projected normal distribution, studied in directional statistics, see for instance \cite{WangGelfand2013,HernandezStumpfhauser2017}. Continuous-discrete Gaussian filters, such as the continuous-discrete extended, unscented and cubature Kalman filters, are standard, see e.g.\ \cite{Sarkka2013}.

Doob's-$h$-transform for conditioning diffusion processes is well known, see for instance \cite{RogersWilliams_vol2} (Chapter IV-39), as are Markovian bridges to a fixed endpoint \cite{FitzsimmonsPitmanYor1993}. Conditioning a diffusion on (a functional of) its terminal value can equivalently be described by an initial enlargement of the filtration with that value, the guiding term then being the information drift of the enlargement \citep{Jacod1985,MansuyYor2006}; for conditioned SDEs from this point of view see \cite{Baudoin2002}.

 The idea of replacing the unknown Doob-$h$-function by an approximation to define a guided process has been introduced in \cite{schauer2017guided} for the specific setting where $Z_T$ is fully observed and the diffusion is uniformly elliptic. This has been extended to hypo-elliptic diffusions and observing $CZ_T$ for a known matrix $C$ \cite{bierkens2020}. The conditioning considered here is  more challenging. Several works in the recent literature on computational statistics and machine learning target learning Doob's $h$-transform, see for instance \cite{chopin2023}, \cite{bffg} and references therein.

\subsection{Outline}

In Section \ref{sec:constant_velocity} we use the planar nearly constant velocity model to motivate approaching this problem from a continuous time perspective. In Section \ref{sec:general} we derive $g$ and the drift of $Z^\star$, as defined in Equation \eqref{eq:Zcirc}. Theorem \ref{thm:general case g and nabla log g} provides closed form expressions for $g$ and $\nabla_z \log g(t,z)$; the case where the conditional covariance of $X_T$ is isotropic is treated as a corollary.  In Section \ref{sec:block} we show that for many tracking models a block structure holds, under which the guiding term is obtained by computing the first row of a small matrix exponential. Well-known tracking models are presented as examples. We then show that, under this block structure, the conditioned process decomposes into independent radial and angular processes. In Section \ref{sec:ncv} we return to the planar nearly constant velocity model, derive the exact rate at which it hits the ray, and discuss numerical discretisation. In Section \ref{sec:nonlinear} we consider models with  nonlinearity $f$ in the noise-driven components, and show that guided processes built on $g$ are equivalent to the true conditioned process, with an explicit likelihood ratio.  In Section \ref{sec:prior} we incorporate a prior on the initial state. The obtained results can be exploited to construct a sequential Monte Carlo method that is presented in Section \ref{sec:smc}.  The paper concludes with numerical illustrations. The appendix contains proofs of some results stated in the main text.

\subsection{Notation}
Throughout, we will assume $Z_t$ is the full state of the system at time $t$. Depending on the specific model considered, $Z_t=(X_t, V_t)$ or $Z_t=(X_t, V_t, A_t)$, with $X_t$, $V_t$ and $A_t$ denoting position, velocity and acceleration at time $t$ respectively. The dimension of the position is denoted by $p\in \{2,3\}$. For $Z_t=(X_t, V_t)$ we say there are 2 blocks; for $Z=(X_t, V_t, A_t)$ we say there are 3 blocks. The number of blocks will be denoted by $k$, so that $d=kp$.

Throughout, we fix a time $T$ at which we observe the process to be on a ray. We define $\tau(t)=T-t$.

By $\Phi(x)$ we denote the cumulative distribution function of the standard normal distribution, evaluated at $x$. By $\phi(x;\mu,\Sigma)$ we denote the density of the normal distribution with mean vector $\mu$ and covariance matrix $\Sigma$, evaluated at $x$. In case $\mu$ and $\Sigma$ are omitted, $\phi(x)$ denotes the density of the standard normal distribution, evaluated at $x$.

For compatible matrices $A\otimes B$ denotes the Kronecker product of $A$ and $B$. $I_k$ and $0_k$ denote the identity matrix and zero matrix in $\RR^{k\times k}$ respectively. $e_j$ denotes the $j$-th standard basis vector (of $\RR^k$, unless stated otherwise).

\section{Motivation: the planar nearly constant velocity model}\label{sec:constant_velocity}
In this section we first specifically go into the planar nearly constant velocity model, where $Z=(X,V)$ and
\[
\dd X_t = V_t \dd t \qquad  \dd V_t =  \dd W_t,
\]
with $W$ a standard Wiener process in $\RR^2$.
 Rather than observing $X$, the position, at a fixed known time $T$,  we observe  its angle with the positive $x$-axis, up to additive noise. Thus, assume we observe a realisation $y_T$ from
\begin{equation}
Y_T = \theta(X_T) + \varepsilon, \quad \varepsilon \sim \mathcal{N}(0,\sigma_Y^2),
\end{equation}
where $\theta(x) = \atan(x)$. Here, $\atan\colon \RR^2\setminus\{0\} \to (-\pi, \pi]$ measures the angle.  While the latent dynamics of $Z_t$ are simple, the filtering problem poses a challenging question which can be approached in two different ways:
\begin{itemize}
	\item by discretising the state dynamics and then conditioning;
	\item by conditioning the continuous process and subsequently discretising.
\end{itemize}
Below we argue why we  believe the second approach is preferable.

\subsection{Problems in the  discrete time formulation}
Using Euler discretisation, the state dynamics can be written as
\[
	X_t  = X_{t-1}+V_{t-1} \qquad 
	V_t = V_{t-1} + \eta_t,\]
with $\eta_t\simiid N(0,\sigma^2 I_2)$.
Let $Z_t=(X_t, V_t)$.
The discrete-time process evolves according to the Markov kernel \[ P_t(z_{t-1}, \dd z_t) = \delta_{x_{t-1}+v_{t-1}}(\dd x_t) \phi(v_t; v_{t-1}, \sigma^2 I_2) \dd v_t. \]
Suppose we aim for particle filtering using the guided particle filter (\cite{Chopin2020introduction}, chapter 10).  The ``optimal'' guided mutation kernel evolves particles according to
\[ M_t^{\rm opt}(z_{t-1}, \dd z_t) = \frac{ P_t(z_{t-1}, \dd z_t) f_t(y_t \mid z_t)}{\int  P_t(z_{t-1}, \dd z_t) f_t(y_t \mid z_t)} ,  \]
with $f_t$ the observation density. Note that since  $x_t$ is concentrated on $x_{t-1}+v_{t-1}$ and $y_t$ only depends on $x_t$ we have  $M_t^{\rm opt}=P_t$. Thus the optimal guided particle filter equals the bootstrap filter, which fails when $\sigma_Y$ is small. 
As a patch,  in exercise 10.1 of \cite{Chopin2020introduction} it is suggested to consider the ``lag-1''-smoothing problem where $z_t$ is sampled conditional on $(z_{t-1}, y_{t+1})$ rather than $(z_{t-1}, y_t)$. The real cause lies in applying (Euler) discretisation to the hypo-elliptic SDE system. This motivates our point of view of considering the latent process as a continuous-time process.

\subsection{The continuous time problem}

The law of $X_T \mid Z_t=z$ is Gaussian with mean
\begin{align*}
	m(t,z)&:= \EE[X_T \mid Z_t=z] = x+\tau(t)v\\
	\Sigma(t)&:= \mbox{Cov}(X_T \mid Z_t=z) = \frac13 \tau(t)^3 I_2,
\end{align*}
where $z=(x,v)$ and $\tau(t)=T-t$. The Doob-$h$-function for conditioning the diffusion process on $Y_T$ is given by
\[ g^{\sigma_Y}(t,z) = \int \phi(y_T; \theta(x_T), \sigma_Y^2) \phi(x_T; m(t,z), \Sigma(t)) \dd x_T. \]
Switching to polar coordinates gives 
\begin{align*}
	g^{\sigma_Y}(t,z) =
\int_{\psi\in (-\pi,\pi]} \phi(y_T; \psi, \sigma_Y^2) \left(	\int_{R>0} R \phi(R [\cos \psi, \sin \psi]^\T; m(t,z), \Sigma(t))  \dd R\right) \dd \psi, 
\end{align*}
If $\sigma_Y \downarrow 0$, this tends to   \[ g(t,z)= \int_{R>0} R \phi(Ru; m(t,z), \Sigma(t)) \dd R,\]
where $u=[\cos y_T, \sin y_T]^\T$.

\section{Linear model ($f\equiv 0$)}\label{sec:general}
In this section we derive closed form expressions for $g$ and $\nabla_z g(t,z)$. 
Assume the position is given by $X_t = CZ_t$ for a fixed matrix $C\in \RR^{p\times d}$.  The  observation is given by  the unit-vector $u\in \RR^p$. 
\begin{defn}
	The matrix valued function $L$ is defined by $L(t):= Ce^{B\tau(t)}$.
\end{defn}

\begin{lem}\label{lem:condmeancov}  We have
\[ X_T \mid Z_t=z \sim N(m(t,z),\Sigma(t)),\]
with
\[ m(t,z)=\EE[X_T\mid Z_t=z]=L(t)z \]
and
\[ \Sigma(t)= \int_t^T L(s) Q L(s)^\T \dd s  .\]
\end{lem}
\begin{proof}
This follows from the standard fact that
\[ Z_T \mid Z_t=z \sim N\left(e^{B\tau(t)} z, \int_t^T e^{B(T-s)}Qe^{B^\T (T-s)} \dd s\right),\]
upon multiplying by $C$ and using $L(s)=Ce^{B(T-s)}$.
\end{proof}
Throughout, we assume that $\Sigma(t)$ is positive definite for all $t<T$. This holds if and only if
\begin{equation}\label{eq:rank-condition}
  \operatorname{rank}\big[\,Cq,\ CBq,\ CB^2q,\ \ldots,\ CB^{d-1}q\,\big]=p .
\end{equation}
See for instance \cite{HermesLaSalle1969} (page 74). 
For the models with block structure that we consider in Section \ref{sec:block} positive definiteness holds by Lemma \ref{lem:kron}(3).

 \begin{defn}\label{def:g}
For $p\ge2$ the guiding function is defined by
\[ g(t,z) = \int_0^\infty R^{p-1} \phi(R u; m(t,z), \Sigma(t)) \dd R. \]
 	 \end{defn}
The term $R^{p-1}$ is the Jacobian of polar coordinates $x=Ru$, with $R>0$ and $u\in S^{p-1}$:
$\dd x=R^{p-1}\dd R\,\varsigma(\dd u)$, where $\varsigma$ is the surface measure on $S^{p-1}$; these are the usual
polar coordinates for $p=2$ and spherical coordinates for $p=3$. Hence $u\mapsto g_u(t,z)$ is the density of the
direction $X_T/\|X_T\|$ given $Z_t=z$ with respect to $\varsigma$; for $p=2$, writing $u=(\cos\theta,\sin\theta)$, it
is the density of the angle $\theta$.
We will derive a simplified expression for the guiding function, as well as the gradient of its logarithm. For that, the following definition is useful.
\begin{defn}
For $n\ge 0$ define the moments of a half-line-truncated normal
distribution by
\[
  M_n(\mu,\sigma) := \int_0^\infty z^n\,\phi(z;\mu,\sigma^2)\dd z .
\]
\end{defn}
In particular,
\begin{equation}\label{eq: M0 and M1}
  M_0(\mu,\sigma)=\Phi(\mu/\sigma), \qquad
  M_1(\mu,\sigma)=\mu\,\Phi(\mu/\sigma)+\sigma\,\varphi(\mu/\sigma).
\end{equation}
A recursion for $M_n$ and its derivative with respect to $\mu$ are given in Lemmas \ref{lem:recursion} and \ref{lem:score} in the appendix.

\begin{defn}
	Define the inner-product induced by $\Sigma(t)$ by $\langle x, y\rangle_{t} = x^\T \Sigma(t)^{-1} y$. Set $\|x\|_t^2 = \langle x, x\rangle_t$.
\end{defn}
\begin{thm}\label{thm:general case g and nabla log g}
Let $g$ be as  defined in Definition \ref{def:g}.
If we define
\[ \xi_t = \frac1{\|u\|_t} \qquad  \mu(t,z)=\frac{\langle m(t,z), u\rangle_t}{\|u\|_t^2} \qquad \Pi_t= I_p-\frac{uu^\T \Sigma(t)^{-1}}{\|u\|_t^2}, \]
then
\[ \log g(t,z) =  C(t) -\frac12 \|\Pi_t m(t,z)\|_t^2 + \log M_{p-1}(\mu(t,z), \xi_t), \]
where $C(t)=(-p/2+1/2)\log (2\pi) -\frac12 \log |\Sigma(t)| - \log  \|u\|_t$ and
\[
  \nabla_z \log g(t,z)
  = L(t)^\T \Sigma(t)^{-1}
  \left[
    (p-1) \xi_t^2\,\frac{M_{p-2}(\mu(t,z),\xi_t)}{M_{p-1}(\mu(t,z),\xi_t)}\,  u
    \;-\; \Pi_t m(t,z)
  \right].
\]
\end{thm}
\begin{proof}
In the derivation we omit dependence on $t$ and $z$ where convenient.
The definitions of $\mu$ and $\Pi_t$ imply that
 $m=\mu u +\Pi_t m$. As $\langle \Pi_tm, u\rangle_t=0$ by construction,  this implies, by the Pythagorean theorem in the inner product $\langle\cdot,\cdot\rangle_t$,
\[ \|Ru-m\|_t^2 =(R-\mu)^2 \|u\|_t^2 + \|\Pi_t m\|_t^2.\]
This gives
\begin{align*}
	g(t,z)&= \int_0^\infty R^{p-1} (2\pi)^{-p/2} |\Sigma(t)|^{-1/2} \exp\left(-\frac12\|Ru-m\|^2_t\right) \dd R\\ &= (2\pi)^{-p/2} |\Sigma(t)|^{-1/2}(2\pi\xi_t^2)^{1/2} \exp\left(-\frac12 \|\Pi_t m\|_t^2\right) \int_0^\infty R^{p-1} \phi(R;\mu,\xi_t^2)\dd R\\
	&= e^{C(t)} \exp\left(-\frac12 \|\Pi_t m\|_t^2\right) M_{p-1}(\mu, \xi_t).
\end{align*}
Taking the logarithm gives the first claim.
For deriving the gradient of the log of $g(t,z)$, define
\[ \Psi(t,y) = -\frac1{2} \|\Pi_t y\|_t^2 + \log M_{p-1}(\xi_t^2\langle u, y\rangle_t , \xi_t) + C(t), \]
then $\log g(t,z)=\Psi(t,m(t,z))$. The matrix $\Pi_t$ is idempotent and self-adjoint with respect to $\langle\cdot,\cdot\rangle_t$, that is, $\Sigma(t)^{-1}\Pi_t=\Pi_t^\T\Sigma(t)^{-1}$. Hence
\[ \nabla_y\left(-\frac12\|\Pi_t y\|_t^2\right) = -\Pi_t^\T\Sigma(t)^{-1}\Pi_t y = -\Sigma(t)^{-1}\Pi_t^2 y = -\Sigma(t)^{-1}\Pi_t y.\]
Together with Lemma \ref{lem:score} this gives
\[ \nabla_y \Psi(t,y) = \Sigma(t)^{-1}\left(-\Pi_t y + (p-1) \xi_t^2 \frac{M_{p-2}(\xi_t^2\langle u, y\rangle_t, \xi_t)}{M_{p-1}(\xi_t^2\langle u, y\rangle_t, \xi_t)}u\right). \]
By the chain rule
\[ \nabla_z \log g(t,z)= (D_z m(t,z))^\T\left. \nabla_y \Psi(t,y)\right|_{y=m(t,z)}. \]
By Lemma \ref{lem:condmeancov} $D_z m(t,z)=L(t)$. The result follows since $\xi_t^2\langle u, m(t,z)\rangle_t=\mu(t,z)$.
\end{proof}

\subsection{The isotropic case}
In many models from tracking, the following simplifying assumption holds. 
\begin{ass}\label{ass:simplifying}
$B$, $Q$ and $C$ are such that $\Sigma(t)=\sigma(t)^2  I_p$.
\end{ass}
Under Assumption \ref{ass:simplifying}, $\langle x,y\rangle_t=\sigma(t)^{-2}\langle x,y\rangle$, and therefore $\xi_t=\sigma(t)$,
\[ \mu(t,z) = \langle m(t,z),u\rangle \qquad \text{and} \qquad \Pi_t=\Pi:=I_p-uu^\T, \]
the orthogonal projection matrix onto $\mbox{span}\{u\}^\perp$.  Define
\[
  e_\perp(t,z) = \frac{\Pi m(t,z)}{\|\Pi m(t,z)\|}.
\]
Note that $e_\perp$, $\mu$, $\Pi$ all depend on $u$, though this is omitted from the notation as the vector $u$ is considered fixed.
We often omit the arguments $(t,z)$ for notational convenience.
 Note that for the special case $p=2$, we have $\Pi m = \langle m, u_\perp\rangle u_\perp$ with  $u_\perp$ a unit vector  orthogonal to $u$. In that case
 \[ e_\perp = \frac{\langle m, u_\perp\rangle}{|\langle m, u_\perp\rangle|} u_\perp = \mbox{sign}(\langle m, u_\perp\rangle) u_\perp. \]
Theorem \ref{thm:general case g and nabla log g} then specialises as follows.

\begin{cor}\label{thm:gradlog_g} Under Assumption \ref{ass:simplifying},
for $g$ defined in Definition \ref{def:g} we have
\[ \log g(t,z) =  -\frac1{2\sigma(t)^2} \|\Pi m(t,z)\|^2 + \log  M_{p-1}(\mu(t,z), \sigma(t)) + C(t), \] where $C(t)=\frac{1-p}{2}\log(2\pi \sigma(t)^2)$, and
\[
  \nabla_z \log g(t,z)
  = L(t)^\T
  \left[
    (p-1)\,\frac{M_{p-2}(\mu(t,z),\sigma(t))}{M_{p-1}(\mu(t,z),\sigma(t))}\, u
    \;-\; \frac{\|\Pi m(t,z)\|}{\sigma(t)^2}\, e_\perp(t,z)
  \right].
\]
\end{cor}
\begin{proof}
Substitute $\Sigma(t)=\sigma(t)^2I_p$, $\xi_t=\sigma(t)$ and $\Pi_t=\Pi$ in Theorem \ref{thm:general case g and nabla log g}, and rewrite $\Pi m(t,z)= \|\Pi m(t,z)\| e_\perp(t,z)$.
\end{proof}

\section{Linear model with a block structure}\label{sec:block}
In this section we specialise our result to a specific case. 
For tracking models, often the following assumption is satisfied.
\begin{ass}[Block structure]\label{ass:kron}
There is a $k\ge2$ with $d=kp$ such that, writing $Z_t$ in $k$ blocks
of size $p$,
\[
  B = \hat B\otimes I_p, \qquad q = e_k\otimes I_p, \qquad C = e_1^\top\otimes I_p,
\]
for some $\hat B\in\RR^{k\times k}$ with $e_1^\T\hat B^je_k\neq0$ for some $j\in\{0,\ldots,k-1\}$.
Thus the noise enters only the last block, isotropically in $\RR^p$, the coupling
coefficients between blocks are scalars applied identically to every spatial coordinate,
and the noise reaches the position. The latter condition holds in particular when the
model is a chain of $k$ integrators (position, velocity, and possibly higher derivatives),
that is, when $\hat B=\sum_{j=1}^{k-1}e_je_{j+1}^\T+e_k\hat B_{k,:}$ for a row vector $\hat B_{k,:}\in\RR^{1\times k}$.
\end{ass}
For $z=(z^{(1)},\ldots,z^{(k)})$ split into $p$-dimensional blocks and $a\in\RR^k$ we write $a^\T z:=\sum_{j=1}^k a_j z^{(j)}$.

\subsection{The guiding term}\label{sec:block-guiding}

\begin{lem}\label{lem:kron}
Impose Assumption \ref{ass:kron}. Define the vector $\hat{c}(s)$ by
\[\hat c(s)^\top := e_1^\top e^{\hat Bs}. \] Let $\hat{c}_j(s)$ denote the $j$-th element of $\hat{c}(s)$ and write $\tau=\tau(t)$.
We have
\begin{enumerate}
	\item $L(t)= \hat{c}(\tau)^\top \otimes I_p$;
	\item $QL(t)^\T = \hat{c}_k(\tau) q$;
	\item $\Sigma(t)=\sigma(t)^2 I_p$ with \[
  \sigma(t)^2 = \int_0^{\tau} \hat c_k(s)^2\,\mathrm ds.
\]
Moreover, $\sigma(t)^2>0$ for all $t<T$, so the standing assumption that $\Sigma(t)$ is positive definite for $t<T$ holds automatically. In particular, Assumption \ref{ass:kron} implies Assumption \ref{ass:simplifying};
\item if $z=(z^{(1)}, z^{(2)},\ldots, z^{(k)})$ is split into $p$-dimensional blocks, then
 \[ \mu(t,z)=\sum_{j=1}^k \hat{c}_j(\tau) \langle u, z^{(j)}\rangle. \]
 In particular $\mu(t,z)$ depends on $z$ only via $(I_k\otimes uu^\T)z$;
\item for any $P\in\RR^{p\times p}$, $(I_k \otimes P)\, B = B\, (I_k \otimes P)$ and $(I_k\otimes P)\,q = q\,P$.
\end{enumerate}
\end{lem}
The proof is given in the appendix.
As a consequence, the guiding term only acts on the last block and is the same for all models satisfying Assumption \ref{ass:kron}, up to the scalar $\hat c_k(\tau)$ and the functions $\mu$ and $\sigma$.

\begin{cor}\label{cor:guided-block}
Under Assumption \ref{ass:kron}, write $Z_t^\star=(Z_t^{\star(1)},\ldots,Z_t^{\star(k)})$ in blocks of size $p$ and define
\[ \alpha(t,z)= (p-1)\,\frac{M_{p-2}(\mu(t,z),\sigma(t))}{M_{p-1}(\mu(t,z),\sigma(t))}\, u
    \;-\; \frac{\|\Pi m(t,z)\|}{\sigma(t)^2}\, e_\perp(t,z) . \]
Then the conditioned process $Z^\star$ of \eqref{eq:Zcirc} satisfies
\[ \dd Z_t^\star = \left[BZ_t^\star + \hat c_k(\tau(t))\,q\,\alpha(t,Z_t^\star)\right]\dd t + q\,\dd W_t^\star. \]
If $\hat B$ is a chain of integrators, this reads blockwise
$\dd Z_t^{\star(j)} = Z_t^{\star(j+1)}\dd t$ for $j<k$ and
$\dd Z_t^{\star(k)} = [\hat B_{k,:}Z_t^\star + \hat c_k(\tau(t))\,\alpha(t,Z_t^\star)]\dd t + \dd W_t^\star$, where $\hat B_{k,:}Z_t^\star:=\sum_{j=1}^k\hat B_{kj}Z^{\star(j)}_t$.
\end{cor}
\begin{proof}
By Corollary \ref{thm:gradlog_g}, $r(t,z)=L(t)^\T\alpha(t,z)$, so $Qr(t,z)=QL(t)^\T\alpha(t,z)=\hat c_k(\tau)\,q\,\alpha(t,z)$ by Lemma \ref{lem:kron}(2). Since $q=e_k\otimes I_p$, this vector is nonzero only in block $k$, where it equals $\hat c_k(\tau)\alpha(t,z)$. The blockwise form follows by writing out $BZ_t^\star=(\hat B\otimes I_p)Z_t^\star$.
\end{proof}

Thus, under the block structure assumption we can easily identify the dynamics of the conditioned process by simply identifying $\hat B$ and computing the first row of $e^{\hat B s}$.

\subsection{Examples}\label{sec:examples}
All examples below satisfy Assumption~\ref{ass:kron} with $\hat B$ a chain of integrators; $\hat B_{k,:}$ denotes the last row of $\hat B$. The relevant quantities are summarised in Table \ref{tab:examples}.

\begin{table}
\centering\small
\begin{tabular}{lcccll}
\toprule
Model & $p$ & $k$ & $\hat B_{k,:}$ & $\hat c(s)^\T$ & $\sigma(t)^2$ \\
\midrule
Nearly-constant-velocity & 2 & 2 & $(0,0)$ & $(1,\ s)$ & $\tau^3/3$ \\
Singer & 2 & 3 & $(0,0,-\vartheta)$ & $(1,\ s,\ \zeta_\vartheta(s))$ & $\int_0^\tau \zeta_\vartheta(s)^2\dd s$ \\
Nearly-constant-acceleration & 2 & 3 & $(0,0,0)$ & $(1,\ s,\ s^2/2)$ & $\tau^5/20$ \\
Constant-turn & 3 & 3 & $(0,-\omega^2,0)$ & $(1,\ \sin(\omega s)/\omega,\ k_\omega(s))$ & $\int_0^\tau k_\omega(s)^2\dd s$ \\
\bottomrule
\end{tabular}
\caption{Block structure of the example models. Here $\tau=\tau(t)$; $\zeta_\vartheta$ and $k_\omega$ are defined in Examples \ref{ex:singer} and \ref{ex:ct}. By Corollary \ref{cor:guided-block}, the guiding term is $\hat c_k(\tau)\vartheta(t,z)$ in the last block, with $\hat c_k$ the last entry of $\hat c$.}\label{tab:examples}
\end{table}

\begin{ex}[Planar nearly-constant-velocity model]\label{ex:Nearly-constant-velocity}
Here $Z=(X,V)$ and
\[
\dd X_t = V_t \dd t \qquad  \dd V_t =  \dd W_t,
\]
which corresponds to \eqref{eq:Z} with $f\equiv0$ and
\[
  B=\begin{pmatrix}0_2&I_2\\0_2 &0_2\end{pmatrix}, \qquad
  q=\begin{pmatrix}0_2\\I_2\end{pmatrix}.
\]
Thus $p=2$, $k=2$ and $\hat{B} =\Bm 0 & 1 \\ 0 & 0\Em$.
Since $\hat{B}$ is nilpotent, $e^{\hat{B} s}=I+\hat{B} s$, giving $\hat{c}(s)^\T= \Bm 1 & s\Em$.
Therefore $QL(t)^\T = \tau(t) q$, $\sigma(t)^2=\int_{0}^{\tau(t)} s^2 \dd s = \tau(t)^3/3$ and $\mu(t,z)=\langle x+\tau(t) v, u\rangle$.
The guiding term is nonzero only in the
$V$-coordinate and equal to
\[
  \tau\left[\frac{M_0(\mu,\sigma)}{M_1(\mu,\sigma)}\, u
  \;-\; \frac{\|\Pi m\|}{\sigma^2}\, e_\perp\right].
\]
This model is studied in detail in Section \ref{sec:ncv}.
\end{ex}

\begin{ex}[Planar Singer model]\label{ex:singer} Here,
$Z=(X,V,A)$ and
\[
\dd X_t = V_t \dd t \qquad  \dd V_t = A_t \dd t \qquad  \dd A_t = -\vartheta A_t\dd t +  \dd W_t,
\]
which corresponds to \eqref{eq:Z} with $f\equiv0$ and
\[
  B=\begin{pmatrix}0_2&I_2 & 0_2\\0_2 & 0_2 &I_2\\ 0_2& 0_2& -\vartheta I_2\end{pmatrix}, \qquad
  q=\begin{pmatrix}0_2\\ 0_2\\I_2\end{pmatrix}.
\]
Thus $p=2$, $k=3$ and $\hat B_{k,:}=(0,0,-\vartheta)$.
The matrix  $\hat B$ is not nilpotent, but $\hat c(s)^\T=\big(1,\ s,\
\zeta_\vartheta(s)\big)$ with
\[ \zeta_\vartheta(s)=  \frac{\vartheta s - 1 + e^{-\vartheta s}}{\vartheta^2},
\]
as is verified directly from the ordinary differential equation $\frac{\dd}{\dd s}\hat c(s)^\T=\hat c(s)^\T\hat B$, $\hat c(0)=e_1$.
This gives  $m(t,z)=x+\tau v+\zeta_\vartheta(\tau)a$ and, by Lemma \ref{lem:kron},
\[
\Sigma(t) = \sigma(t)^2 I_2 \qquad \text{with} \quad \sigma(t)^2=\int_0^{\tau(t)}  \zeta_\vartheta(s)^2 \dd s.
\]
 Note that $\zeta_\vartheta(s) = s^2/2 + \scr{O}(s^3)$ as $s \downarrow 0$ and therefore $\sigma(t)^2 = \frac{\tau^5}{20} +\scr{O}(\tau^6)$, $\tau \downarrow 0$.
 The guiding term is
nonzero only in the $A$-coordinate,
\[
  \zeta_\vartheta(\tau)\left[\frac{M_0(\mu,\sigma)}{M_1(\mu,\sigma)}\, u
  \;-\; \frac{\|\Pi m\|}{\sigma^2}\, e_\perp\right] .
\]
The specific case  $\vartheta=0$ is referred to as the nearly-constant-acceleration model. Then $\zeta_0(s)=s^2/2$ and $\sigma(t)^2=\tau^5/20$.
\end{ex}

\begin{ex}[Constant-turn model in $\RR^3$]\label{ex:ct}
Here, $Z=(X,V,A)$ and
\[ \dd X_t = V_t \dd t\qquad  \dd V_t = A_t \dd t\qquad  \dd A_t = -\omega^2 V_t\dd t +  \dd W_t,\]
which corresponds to \eqref{eq:Z} with $f\equiv0$ and
\[
  B=\begin{pmatrix}0_3&I_3 & 0_3\\0_3 & 0_3 &I_3\\ 0_3& -\omega^2 I_3& 0_3\end{pmatrix}, \qquad
  q=\begin{pmatrix}0_3\\ 0_3\\ I_3\end{pmatrix}.
\]
Thus $p=3$, $k=3$ and $\hat B_{k,:}=(0,-\omega^2,0)$.
 Define
\[ k_\omega(s) = \frac{1-\cos(\omega s)}{\omega^2}.\]
Then $\hat c(s)^\T=\big(1,\ \tfrac{\sin(\omega s)}{\omega},\
k_\omega(s)\big)$, and we get
$m(t,z)=x+\tfrac{\sin(\omega\tau)}{\omega}v+k_\omega(\tau)a$. Moreover, by Lemma \ref{lem:kron}, Assumption \ref{ass:simplifying} is satisfied with $\sigma(t)^2=\int_0^{\tau(t)}k_\omega(s)^2 \dd s$.
 The guiding
term is nonzero only in the $A$-coordinate and equal to
\[
  k_\omega(\tau)\left[2\,\frac{M_1(\mu,\sigma)}{M_2(\mu,\sigma)}\, u
  \;-\; \frac{\|\Pi m\|}{\sigma^2}\, e_\perp\right].
\]
\end{ex}

\subsection{Radial and angular processes}\label{sec:radial-angular}

Under Assumption \ref{ass:kron}, the conditioned process decomposes into independent angular and radial conditioned processes. We start with defining these processes. {\it Warning on notation: }The angular process is denoted by ``A'', not to be confused with acceleration. 
\begin{defn}
Define the angular conditioned process by
\[
  A^{\star}_t := (I_k\otimes\Pi)\,Z_t^\star \in\RR^{kp}, \qquad
  A^{\star}_t=(A_t^{\star(1)},\ldots,A_t^{\star(k)}), \quad A_t^{\star(j)}=\Pi\,Z_t^{\star(j)}\in u^\perp\subset\RR^p.
\]
In particular $A_t^{\star(1)} = \Pi X_t^\star = \Pi C Z_t^\star$.
Let $\Pi^{\rad}:=uu^\T$. Define the radial conditioned process by 
\[
  R^{\star}_t := (I_k\otimes\Pi^{\rad})\,Z_t^\star \in\RR^{kp}, \qquad
  R^{\star}_t=(R_t^{\star(1)},\ldots,R_t^{\star(k)}), \quad R_t^{\star(j)}\in\mathrm{span}(u)\subset\RR^p.
\]
\end{defn}

\begin{thm}\label{thm:closed-system}
Under Assumption~\ref{ass:kron}, $A^{\star}_t$ satisfies the
closed, autonomous linear SDE
\[ \dd A^{\star}_t = \left[BA^{\star}_t - \frac{\hat c_k(\tau)}{\sigma(t)^2}\,q\,\hat c(\tau)^\T A^{\star}_t\right]\dd t + q\,\dd W^{\star\ang}_t , \]
where $W^{\star\ang}_t:=\Pi W_t^\star$ is a (degenerate, rank-$(p-1)$)
Brownian motion valued in $u^\perp\subset\RR^p$, with $\mathrm{Cov}(W^{\star\ang}_t)=\Pi\,t$.

 $R^{\star}_t$ satisfies the closed, autonomous (but nonlinear) SDE
\[ \dd R^{\star}_t = \left[BR^{\star}_t + \hat c_k(\tau)\,(p-1)\frac{M_{p-2}(\mu,\sigma)}{M_{p-1}(\mu,\sigma)}\,q\,u\right]\dd t + q\,\dd W^{\star\rad}_t , \]
where $\mu=\mu(t,R^{\star}_t)=\sum_{j=1}^k\hat c_j(\tau)\langle u,R_t^{\star(j)}\rangle$, $\sigma=\sigma(t)$ and
$W_t^{\star\rad}:=\Pi^\rad W_t^\star$ is a (degenerate, rank-$1$) Brownian motion valued in
$\mathrm{span}(u)\subset\RR^p$, with $\mathrm{Cov}(W^{\star\rad}_t)=\Pi^\rad\, t$. Moreover $W_t^{\star\rad}$ is independent of  $W^{\star\ang}_t$, so $R^{\star}_t$ and $A^{\star}_t$ are driven by
independent noise and neither system's SDE references the other.

\end{thm}

The proof of Theorem \ref{thm:closed-system} is given in the appendix. Since $Z_t^\star=R^{\star}_t+A^{\star}_t$, the conditioned process can be simulated by simulating the linear SDE for $A^{\star}$ and the one-dimensional-noise SDE for $R^{\star}$ independently.

\section{The planar nearly constant velocity model in detail}\label{sec:ncv}
We return to the model of Section \ref{sec:constant_velocity}. By Example \ref{ex:Nearly-constant-velocity}, $p=k=2$,
$\hat c(s)^\T=(1,s)$, $m(t,z)=x+\tau v$ and $\sigma(t)^2=\tau^3/3$, with $\tau=\tau(t)$. Let
$u=[\cos y_T,\sin y_T]^\T$ and $u_\perp=[-\sin y_T,\cos y_T]^\T$, so that $\Pi=u_\perp u_\perp^\T$ and
$\Pi^\rad=uu^\T$, and let $F(\mu,\sigma)=M_0(\mu,\sigma)/M_1(\mu,\sigma)$. The angular and radial processes of
Section \ref{sec:radial-angular} are $A^\star_t=(X^{\star\ang}_t,V^{\star\ang}_t)$ and
$R^\star_t=(X^{\star\rad}_t,V^{\star\rad}_t)$, with $X^{\star\ang}_t=\Pi X^\star_t$, $V^{\star\ang}_t=\Pi V^\star_t$,
$X^{\star\rad}_t=\Pi^\rad X^\star_t$ and $V^{\star\rad}_t=\Pi^\rad V^\star_t$. They take values in
$\mathrm{span}(u_\perp)$ and $\mathrm{span}(u)$, respectively, so each is effectively one-dimensional.

\begin{cor}\label{cor:ang radial processes}
The angular process satisfies the linear SDE
\begin{align*}
	\dd X^{\star\ang}_t &= V^{\star\ang}_t \dd t \\
	\dd V^{\star\ang}_t &= \dd W^{\star\ang}_t -\frac3{\tau^2} X^{\star\ang}_t\dd t -\frac3{\tau} V^{\star\ang}_t \dd t ,
\end{align*}
and the radial process satisfies the nonlinear SDE
\begin{align*}
	\dd X^{\star\rad}_t &= V^{\star\rad}_t \dd t \\
	\dd V^{\star\rad}_t &= \dd W^{\star\rad}_t + \tau F\big(\langle u, X^{\star\rad}_t + \tau V^{\star\rad}_t\rangle, \sigma(t)\big)\, u \dd t ,
\end{align*}
where $W^{\star\ang}=\Pi W^\star$ and $W^{\star\rad}=\Pi^\rad W^\star$ are independent.
\end{cor}
\begin{proof}
This is the special case $p=k=2$ of Theorem \ref{thm:closed-system}, with $\hat c_k(\tau)=\tau$,
$\hat c(\tau)^\T A^\star_t=X^{\star\ang}_t+\tau V^{\star\ang}_t$ and $\tau/\sigma(t)^2=3/\tau^2$.
\end{proof}

The following theorem describes the rate at which $X^{\star\ang}_t\to0$ as $t\uparrow T$.

\begin{thm}\label{thm:hits-ray}
For the nearly-constant-velocity model,
\[  \limsup_{t\uparrow T} \frac{\|X^{\star\perp}_t- \tfrac12\tau(t)(W^{\star\perp}_t+C_0)\|}{\tau(t)\sqrt{\tau(t) \log \log (1/\tau(t))}} = \frac{1}{\sqrt6}, \qquad \text{a.s.} \]
where $C_0 = V_0^{\star\perp}+3X_0^{\star\perp}/T\in\mathrm{span}(u_\perp)$.
In particular, the angular process converges to zero almost surely and thus $X_T^\star\in\mathrm{span}(u)$.
\end{thm}
The proof is given in the appendix and depends on the law of the iterated logarithm for Brownian motion.

\subsection{Discretisation of the conditioned process}\label{sec:discr_ncv}
For simulation purposes, rather than discretising $(X_t^{\star\ang}, V_t^{\star\ang})$, we instead discretise $(X_t^{\star\ang}, M_t^{\star\ang})$, where \[M_t^{\star\ang} = \Pi\big(X_t^\star + \tau V_t^\star\big) = X_t^{\star\ang} + \tau V_t^{\star\ang} = \Pi m(t,Z_t^\star).\]
The next lemma shows that the derived system for $M_t^{\star\ang}$ eliminates the $\tau^{-2}$ term, making it more amenable for numerical discretisation. Note that the SDE for the radial part is more ``friendly'' to numerical discretisation.

\begin{lem}\label{lem:discretisation}
The process $(X_t^{\star\ang}, M_t^{\star\ang})$ satisfies the SDE
\begin{equation}\label{eq;ncv-sde}
\begin{split}
	\dd X_t^{\star\ang} & = \frac{M_t^{\star\ang} - X_t^{\star\ang}}{\tau} \dd t\\
	\dd M_t^{\star\ang} & = -3\frac{M_t^{\star\ang}}{\tau} \dd t + \tau \dd W^{\star\ang}_t
\end{split}	
\end{equation}
\end{lem}
\begin{proof}
We simply replace $V_t^{\star\ang}$ by $(M_t^{\star\ang}-X^{\star\ang}_t)/\tau$.
The equation for $X_t^{\star\ang}$ is straightforward. For the second equation, note that on the one hand
\[ \dd \left(\frac{M_t^{\star\ang}-X^{\star\ang}_t}{\tau} \right)= \frac1{\tau^2} (M_t^{\star\ang}-X^{\star\ang}_t) \dd t + \frac1{\tau} \dd (M_t^{\star\ang}-X^{\star\ang}_t)=\frac1{\tau}\dd M_t^{\star\ang}, \]
where the final equality follows upon substituting the SDE for $X_t^{\star\ang}$.
On the other hand, from the SDE for $V_t^{\star\ang}$ we also have
\[  \dd \left(\frac{M_t^{\star\ang}-X^{\star\ang}_t}{\tau}\right) =-\frac{3}{\tau^2} M_t^{\star\ang} \dd t + \dd W_t^{\star\ang}. \]
Combining the two equations gives the SDE for $M_t^{\star\ang}$.
\end{proof}
\begin{lem}\label{lem:impl_euler_ncv}
One implicit Euler discretisation step of \eqref{eq;ncv-sde}, from $t$ to $t+\delta$, is given by
\[
\begin{split}
 (\tau_\delta+3\delta)\,M_{t+\delta}^{\star\ang} &= \tau_\delta\big(M_t^{\star\ang} +\tau(t)(W^{\star\ang}_{t+\delta}-W^{\star\ang}_t)\big),\\
 (\tau_\delta+\delta)\,X_{t+\delta}^{\star\ang} &= \tau_\delta X_{t}^{\star\ang} + \delta\, M_{t+\delta}^{\star\ang},
\end{split}
\]
where $\tau_\delta:=\tau(t+\delta)$. In the last step, $t+\delta=T$, we have $\tau_\delta=0$, so that $M_T^{\star\ang}=0$ and $X_T^{\star\ang}=M_T^{\star\ang}=0$: the discretised process ends exactly on the ray, for every grid.
\end{lem}
This follows directly from the implicit Euler step for \eqref{eq;ncv-sde}, with the drift evaluated at $t+\delta$. 

\section{Guided process for the nonlinear model}\label{sec:nonlinear}

Let $q$ have full column rank, with left inverse $q^+:=(q^\T q)^{-1}q^\T$. Consider the nonlinear model \eqref{eq:Z},
in which the nonlinearity $qf$ lies in the range of the diffusion coefficient. For $f\neq 0$ it is a hard problem to sample from  the process conditioned to hit the ray. 
For a direction $u\in S^{p-1}$, write $g_u$ for the function $g$ of Definition~\ref{def:g}, and define the {\it guided process} as the process $Z^\circ$  solving the SDE in \eqref{eq:guided-nonlinear}.
In this section we show that this process ends on the ray at time $T$. Moreover, we derive the Radon-Nikodym derivative of the law of the conditioned process with respect to the law of the guided process. In case $f\equiv 0$ the guiding function $g_u$ is exact, so the guided process coincides with the
conditioned process $Z^\star$. We will make the following assumption on $f$.  

\begin{ass}\label{ass:dh}
The function $f\colon[0,T]\times\RR^d\to\RR^m$ is measurable, bounded, and locally Lipschitz in $z$, uniformly in $t$.
\end{ass}

\subsection{The guided process ends at the ray}
Let $\Omega=C([0,T],\RR^d)$ with coordinate process $Z$ and natural filtration $(\scr F_t)$. On $\Omega$, denote by $\PP$ the law of the unconditioned process \eqref{eq:Z} and by 
$\mathbb{L}$ the law of the unconditioned process when $f\equiv 0$. For these laws, and for the laws
introduced below, we use the same symbol for the expectation, for instance $\mathbb{L}\xi=\int\xi\dd\mathbb{L}$. Set
$U=X_T/\|X_T\|$. 
Let $\mathbb{L}_u$ be a regular version of the conditional law of $Z$ under $\mathbb{L}$ given
$U=u$; it exists because $\Omega$ is Polish. With $\varsigma$ the surface measure on $S^{p-1}$ (Section~\ref{sec:general}), by definition,
\begin{equation}\label{eq:disintegration}
  \mathbb{L}(A)=\int_{S^{p-1}}\mathbb{L}_u(A)\,g_u(0,z_0)\,\varsigma(\dd u),\qquad A\subseteq\Omega\ \text{Borel},
\end{equation}
and $\mathbb{L}_u$ is concentrated on $\{U=u\}=\{X_T\in\{Ru\colon R>0\}\}$ for $\varsigma$-almost every $u$. The following result relates $\mathbb{L}_u$ to the SDE for $Z^\star$ in Equation \eqref{eq:Zcirc}. 
\begin{lem}\label{lem:conditional-law}
For $\varsigma$-almost every $u$, under $\mathbb{L}_u$ the process
\[
  W_t^\star=q^+\Big(Z_t-z_0-\int_0^t\big[BZ_s+Q\nabla_z\log g_u(s,Z_s)\big]\dd s\Big),\qquad t<T,
\]
is an $(\scr F_t)_{t<T}$-Brownian motion on $[0,T)$, and $Z$ solves \eqref{eq:Zcirc} with driving process $W^\star$.
\end{lem}
The proof is given in the appendix.

\begin{thm}\label{lem:strong-nonlinear}
Under Assumption~\ref{ass:dh}:
\begin{enumerate}
\item Equation \eqref{eq:Z} has a unique strong solution on $[0,T]$.
\item For $\varsigma$-almost every $u$, equation \eqref{eq:guided-nonlinear} has a unique strong solution on $[0,T)$. The limit
$Z^\circ_T:=\lim_{t\uparrow T}Z^\circ_t$ exists almost surely, and the law $\PP^\circ_u$ of $Z^\circ$ on $\Omega$
satisfies $\dd\PP^\circ_u/\dd\mathbb{L}_u=D_T$, where
\[
  D_t=\exp\Big(\int_0^tf(s,Z_s)^\T\dd W_s^\star-\frac12\int_0^t\|f(s,Z_s)\|^2\dd s\Big),
  \qquad D_T:=\lim_{t\uparrow T}D_t>0 .
\]
In particular, $\PP^\circ_u$ and $\mathbb{L}_u$ are equivalent, and the guided proposal ends on the ray:
$CZ^\circ_T\in\{Ru\colon R>0\}$ almost surely.
\end{enumerate}
\end{thm}

\begin{proof}
(i) The drift $Bz+qf(t,z)$ is locally Lipschitz in $z$, uniformly in $t$. By boundedness of $f$,  there exists a $K>0$ such that 
\begin{equation}\label{eq:for strong solution}
	\langle x, B x+ q f(x)\rangle +\mbox{tr}(Q) \le K \left(1+\|x\|^2\right). 
\end{equation} 
Existence of a unique strong solution follows from Theorem 2.3.3 in \cite{reiss_sde_lecture}.

(ii) \emph{Existence of weak solution to \eqref{eq:guided-nonlinear}.} Work under $\mathbb{L}_u$. Since $f$ is bounded, $(D_t)_{t<T}$ is a martingale by
Novikov's condition. Moreover,
\[
  D_t^2=\exp\Big(\int_0^t(2f)^\T\dd W^\star_s-\frac12\int_0^t\|2f\|^2\dd s\Big)\,
        \exp\Big(\int_0^t\|f(s,Z_s)\|^2\dd s\Big),
\]
where we wrote $f$ for $f(s,Z_s)$ in the first factor. The first factor is the stochastic exponential of the bounded
process $2f$, hence a martingale with expectation one by Novikov's condition, and the second factor is at most
$\exp(T\sup\|f\|^2)$. Therefore $\mathbb{L}_u D_t^2\le\exp(T\sup\|f\|^2)$ for all $t<T$, and $(D_t)$ is uniformly integrable. The limit $D_T$
exists almost surely and in $L^1$, and $\mathbb{L}_u D_T=1$. Also $D_T>0$, because the $L^2$-bounded martingale
$\int_0^tf^\T\dd W^\star$ converges almost surely. Let $\QQ$ be the probability measure on $\Omega$ with $\dd\QQ/\dd\mathbb{L}_u=D_T$. By Girsanov's theorem, on
every $[0,t]$ with $t<T$, $\widetilde W=W^\star-\int_0^\cdot f(s,Z_s)\,\dd s$ is a $\QQ$-Brownian motion, and
\[
  \dd Z_s=\big[BZ_s+qf(s,Z_s)+Q\nabla_z\log g_u(s,Z_s)\big]\dd s+q\,\dd\widetilde W_s .
\]
So $(Z,\widetilde W)$ under $\QQ$ is a weak solution of \eqref{eq:guided-nonlinear}, with paths that are continuous
on $[0,T]$.

\emph{Pathwise uniqueness.} Fix $t<T$ and write
$b(s,z)=Bz+qf(s,z)+Q\nabla_z\log g_u(s,z)$ for the drift of \eqref{eq:guided-nonlinear}. By
Assumption~\ref{ass:dh} and Lemma~\ref{lem:g-lipschitz} in the appendix, for every $n$ there is a constant $K_n$ such that
$\|b(s,z)-b(s,z')\|\le K_n\|z-z'\|$ for all $s\in[0,t]$ and $\|z\|,\|z'\|\le n$. Let $Z$ and $Z'$ be two solutions
of \eqref{eq:guided-nonlinear} on $[0,t]$, driven by the same Brownian motion and with $Z_0=Z'_0=z_0$. Since the
diffusion coefficient $q$ is constant, the noise cancels in their difference:
\[
  Z_s-Z'_s=\int_0^s\big(b(r,Z_r)-b(r,Z'_r)\big)\dd r,\qquad s\le t .
\]
Let $\tau_n$ be the first time at which $Z$ or $Z'$ leaves the ball of radius $n$. Then
$\|Z_s-Z'_s\|\le K_n\int_0^s\|Z_r-Z'_r\|\dd r$ for $s\le t\wedge\tau_n$, so $Z=Z'$ on $[0,t\wedge\tau_n]$ by
Gronwall's inequality, path by path. Since $Z$ and $Z'$ have continuous paths, $\tau_n>t$ for all large $n$, and
hence $Z=Z'$ on $[0,t]$ almost surely. This is pathwise uniqueness for \eqref{eq:guided-nonlinear} on $[0,t]$.

\emph{Existence of a strong solution.} By the Yamada--Watanabe theorem \cite[Chapter 5, Corollary 3.23]{KaratzasShreve1991}, weak existence and
pathwise uniqueness give a unique strong solution on $[0,t]$ for every $t<T$. By uniqueness, these solutions are
consistent, so they define a unique strong solution $Z^\circ$ on $[0,T)$. By uniqueness in law, the law of
$(Z^\circ_s)_{s\le t}$ equals $\QQ$ on $\scr F_t$ for every $t<T$. Since $\QQ$-almost every path is continuous at $T$,
the limit $Z^\circ_T$ exists almost surely. Since $\scr F_T=\sigma\big(\bigcup_{t<T}\scr F_t\big)$, it follows that
$\PP^\circ_u=\QQ$. Since $D_T>0$, $\PP^\circ_u$ and $\mathbb{L}_u$ are equivalent. As
$\mathbb{L}_u$ is concentrated on $\{U=u\}=\{CZ_T\in\{Ru\colon R>0\}\}$, so is $\PP^\circ_u$.
\end{proof}

\subsection{Absolute continuity}
For $t<T$, $z\in\RR^d$ and Borel sets $B\subseteq S^{p-1}$, let the function $u\mapsto h_u(t,z)$ be
defined by
\[
  \int_B h_u(t,z)\,\varsigma(\dd u)=P\big(t,z;T,\{z'\colon Cz'/\|Cz'\|\in B\}\big),
\]
where $P$ is the transition kernel of \eqref{eq:Z}. That is, $h_u(t,z)$ is the density, with respect to
the surface measure $\varsigma$ on $S^{p-1}$, of the direction of $X_T$ given $Z_t=z$. It exists by the Radon--Nikodym theorem. Indeed, since $f$ is bounded, Girsanov's theorem shows that the law of
$(Z_s)_{s\in[t,T]}$ given $Z_t=z$ under \eqref{eq:Z} is equivalent to the corresponding law under the linear
model ($f\equiv0$). Hence the law of the direction of $X_T$ given $Z_t=z$ is absolutely continuous with respect to the
corresponding law under the linear model, which has density $g_u(t,z)$ with respect to $\varsigma$ by
Definition~\ref{def:g}.

Similar to our construction of $\mathbb{L}_u$, let $\PP_u$ be a regular version of the conditional law of $Z$ under $\PP$ given $U=u$. By definition,
\begin{equation}\label{eq:disintegration-nonlinear}
  \PP(A)=\int_{S^{p-1}}\PP_u(A)\,h_u(0,z_0)\,\varsigma(\dd u),\qquad A\subseteq\Omega\ \text{Borel}.
\end{equation}
We do not describe $\PP_u$ as the law of the solution of an SDE: its drift would involve $Q\nabla_z\log h_u$, and
nothing is known about the regularity of $h_u$. Instead, Theorem~\ref{prop:ac-nonlinear} characterises $\PP_u$
through its density with respect to the law of the guided process, which involves only the explicit guiding term
$Q\nabla_z\log g_u$ and the number $h_u(0,z_0)$.

\begin{thm}\label{prop:ac-nonlinear}
Under Assumption~\ref{ass:dh}, for $\varsigma$-almost every $u$, $\PP_u$ and $\PP^\circ_u$ are equivalent, and
\begin{equation}\label{eq:weight-nonlinear}
  \frac{\dd\PP_u}{\dd\PP^\circ_u}(Z)=\frac{g_u(0,z_0)}{h_u(0,z_0)}\,
  \exp\Big(\int_0^T f(s,Z_s)^\T q^\T\nabla_z\log g_u(s,Z_s)\,\dd s\Big),
\end{equation}
where $\int_0^T=\lim_{t\uparrow T}\int_0^t$, and the limit exists $\PP^\circ_u$-almost surely.
\end{thm}

\begin{proof}
\emph{Step 1: Likelihood ratio between $\PP$ and $\mathbb{L}$.} By Theorem~1 of \cite{DelyonHu2006} with $b(t,z)=Bz$,
$\sigma=q$ and $h=f$ we obtain
 $\dd \PP = \Psi_T \dd \mathbb{L}$ on $\Omega$, where
\[
\Psi_t= \exp\Big(\int_0^t f(s,Z_s)^\T q^+\big(\dd Z_s-BZ_s\,\dd s\big)-\frac12\int_0^t\|f(s,Z_s)\|^2\dd s\Big)
\]
Here we used that $W=q^+(Z-z_0-\int_0^\cdot BZ_s\,\dd s)$ under $\mathbb{L}$. The stochastic integral is a
functional of the path, defined $\mathbb{L}$-almost surely as the limit in probability of Riemann sums, as in
the proof of Theorem~2 of \cite{DelyonHu2006}. We fix a version that is continuous in the upper limit of integration.

\emph{Step 2: Likelihood ratio between $\PP_u$ and $\mathbb{L}_u$.} Let $A\subseteq\Omega$ and $B\subseteq S^{p-1}$ be Borel sets. By \eqref{eq:disintegration-nonlinear}, Step~1 and \eqref{eq:disintegration},
\begin{align*}
	  \int_B\PP_u(A)\,h_u(0,z_0)\,\varsigma(\dd u)&=\PP(A\cap\{U\in B\})
  =\mathbb{L}\big[\Psi_T\ind_A\ind_B(U)\big]
 \\& =\int_B\mathbb{L}_u\big[\Psi_T\ind_A\big]\,g_u(0,z_0)\,\varsigma(\dd u).
\end{align*}
Taking $A=\Omega$ shows that $h_u(0,z_0)=g_u(0,z_0)\,\mathbb{L}_u[\Psi_T]$ for $\varsigma$-almost every $u$.
This is positive, because $g_u(0,z_0)>0$ by Theorem~\ref{thm:general case g and nabla log g} and $\Psi_T>0$. 
Since the Borel $\sigma$-field of $\Omega$ is countably generated, it follows that for $\varsigma$-almost every $u$
and all Borel sets $A\subseteq\Omega$
\[ \PP_u(A) = \frac{g_u(0,z_0)}{h_u(0,z_0)}\, \mathbb{L}_u\big[\Psi_T\ind_A\big] . \]
Therefore, 
\begin{equation}\label{eq:bayes-nonlinear}
  \frac{\dd\PP_u}{\dd\mathbb{L}_u}=\frac{g_u(0,z_0)}{h_u(0,z_0)}\,\Psi_T .
\end{equation}

\emph{Step 3: $\Psi_T$ under $\mathbb{L}_u$.} Fix $t<T$. On $\scr F_t$, $\mathbb{L}_u$ and
$\mathbb{L}$ are equivalent, so the stochastic integral in $\Psi_t$ is also the stochastic integral under
$\mathbb{L}_u$. Under $\mathbb{L}_u$,
$q^+(\dd Z_s-BZ_s\,\dd s)=\dd W_s^\star+q^\T\nabla_z\log g_u(s,Z_s)\,\dd s$, since $q^+Q=q^\T$. Hence
\[
  \Psi_t=\exp\Big(\int_0^tf^\T q^\T\nabla_z\log g_u\,\dd s\Big)\,D_t,\qquad
  \mathbb{L}_u\text{-almost surely}.
\]
Moreover, $\Psi_t\to\Psi_T$ outside a $\mathbb{L}$-null set. By \eqref{eq:disintegration}, this set is
$\mathbb{L}_u$-null for $\varsigma$-almost every $u$. Since $D_t\to D_T\in(0,\infty)$ by
Theorem~\ref{lem:strong-nonlinear}(2), and $\Psi_T\in(0,\infty)$, the limit
$E_T:=\lim_{t\uparrow T}\exp\big(\int_0^tf^\T q^\T\nabla_z\log g_u\,\dd s\big)$ exists in $(0,\infty)$, and
$\Psi_T=E_T\,D_T$, $\mathbb{L}_u$-almost surely.

\emph{Step 4: conclusion.} By \eqref{eq:bayes-nonlinear}, Step~3 and Theorem~\ref{lem:strong-nonlinear}(2),
\[
  \frac{\dd\PP_u}{\dd\mathbb{L}_u}=\frac{g_u(0,z_0)}{h_u(0,z_0)}\,E_T\,D_T
  \qquad\text{and}\qquad
  \frac{\dd\PP^\circ_u}{\dd\mathbb{L}_u}=D_T>0 .
\]
Hence $\PP_u$ and $\PP^\circ_u$ are equivalent, with density $g_u(0,z_0)E_T/h_u(0,z_0)$, which is
\eqref{eq:weight-nonlinear}.
\end{proof}

Since $f^\T q^\T\nabla_z\log g_u=(b-\tilde b)^\T\nabla_z\log g_u$, where $b$ is the drift of
\eqref{eq:Z} and $\tilde b$ its drift with $f\equiv0$, \eqref{eq:weight-nonlinear} is the familiar weight of a guided proposal
\cite{schauer2017guided}. Compared with the weight of \cite{DelyonHu2006}, who use the Gaussian bridge itself as the
proposal, adding $qf$ to the proposal replaces the stochastic integral by an ordinary integral. The estimator
$g_u(0,z_0)\exp\big(\int_0^T(b-\tilde b)^\T\nabla_z\log g_u\,\dd s\big)$ of the likelihood $h_u(0,z_0)$ is unbiased.

\begin{cor}\label{cor:linear-growth}
Theorem~\ref{lem:strong-nonlinear} and Theorem~\ref{prop:ac-nonlinear} remain valid if, in Assumption~\ref{ass:dh}, boundedness of $f$ is replaced by linear growth: there is a constant $K$ such that $\|f(t,z)\|\le K(1+\|z\|)$ for all $t\in[0,T]$ and $z\in\RR^d$.
\end{cor}

The proof is given in the appendix.

\begin{rem}
The requirement that the diffusivity is constant and  the nonlinearity lies in the range of $q$ is essential. Under these assumptions, Girsanov's theorem allows to derive $\PP_u \ll \mathbb{L}_u$ and $\PP^\circ_u \ll \mathbb{L}_u$ separately. 
However, a nonlinearity acting on the components
without noise, or a state-dependent diffusion coefficient, cannot be removed by Girsanov's theorem. In those cases
matching conditions at the terminal time are needed \cite{schauer2017guided,bierkens2020}, and for conditioning on a
ray they would have to hold along the ray. Thus, for the SDE $\dd Z_t = \left(BZ_t + f(Z_t)\right) \dd t + q(Z_t)\dd W_t$ the problem is much harder.
\end{rem}

\section{Incorporating a prior on the initial state}\label{sec:prior}

Up to this point we have assumed $z_0$ to be known, but this is a rather unrealistic assumption. It turns out that the
posterior of $Z_0$, i.e.\ the distribution of $Z_0$ conditional on the observed direction $u$, can be sampled exactly
for any linear model, i.e.\ \eqref{eq:Z} with $f\equiv0$: it is a mixture of Gaussian distributions over the radius $R$ of $X_T=Ru$, and the
mixing distribution is one-dimensional. Let $L(0)$ and $\Sigma(0)$ be as in Lemma~\ref{lem:condmeancov}, so that
$X_T\mid Z_0=z_0\sim N(L(0)z_0,\Sigma(0))$.

\begin{thm}\label{thm:prior}
Assume $Z_0\sim N(\nu,A)$ with $A\in\RR^{d\times d}$ positive semi-definite. Define
\[
  \bar m=L(0)\nu,\qquad S=L(0)AL(0)^\T+\Sigma(0),\qquad K=AL(0)^\T S^{-1},
\]
and
\[
  \xi=\big(u^\T S^{-1}u\big)^{-1/2},\qquad \mu=\xi^2\,u^\T S^{-1}\bar m .
\]
A sample from $Z_0$ conditional on the observed direction $u$ is obtained by
\begin{enumerate}
\item sampling $R$ from the density $R^{p-1}\phi(R;\mu,\xi^2)/M_{p-1}(\mu,\xi)$ on $(0,\infty)$;
\item sampling $Z_0\sim N\big(\nu+K(Ru-\bar m),\ A-KL(0)A\big)$.
\end{enumerate}
\end{thm}
\begin{proof}
The matrix $S$ is positive definite, since $\Sigma(0)$ is. By Definition~\ref{def:g}, the posterior density of $Z_0$
is proportional to
\[
  \phi(z_0;\nu,A)\,g(0,z_0)=\int_0^\infty R^{p-1}\,\phi(z_0;\nu,A)\,\phi\big(Ru;L(0)z_0,\Sigma(0)\big)\dd R .
\]
(If $A$ is singular, densities in $z_0$ are with respect to Lebesgue measure on the support of $N(\nu,A)$.) The
integrand is $R^{p-1}$ times the joint density of $(Z_0,X_T)$ at $(z_0,Ru)$, where $(Z_0,X_T)$ is jointly Gaussian with
$X_T\sim N(\bar m,S)$ and $\mathrm{Cov}(Z_0,X_T)=AL(0)^\T$. Factorising this joint density into the marginal of
$X_T$ and the conditional of $Z_0$ given $X_T$ gives
\[
  \phi(z_0;\nu,A)\,g(0,z_0)=\int_0^\infty R^{p-1}\,\phi(Ru;\bar m,S)\,
  \phi\big(z_0;\nu+K(Ru-\bar m),A-KL(0)A\big)\dd R .
\]
As in the proof of Theorem~\ref{thm:general case g and nabla log g}, with $\Sigma(t)$ replaced by $S$,
$\phi(Ru;\bar m,S)$ is proportional to $\phi(R;\mu,\xi^2)$ as a function of $R$. Hence the posterior of $Z_0$ is the
mixture of the Gaussian distributions in (2) over $R$, with mixing density proportional to
$R^{p-1}\phi(R;\mu,\xi^2)$ on $(0,\infty)$, whose normalising constant is $M_{p-1}(\mu,\xi)$.
\end{proof}

The variable $R$ is the distance of $X_T$ from the origin along the observed ray: jointly, $(Z_0,Ru)$ is a draw of
$(Z_0,X_T)$ conditional on $X_T/\|X_T\|=u$. Step (i) is a one-dimensional draw, which can be done exactly by rejection
sampling, since the density is log-concave on $(0,\infty)$. Step (ii) is a draw from a multivariate normal distribution. Starting the
conditioned process \eqref{eq:Zcirc} from the sampled $Z_0$ gives a draw from the smoothing distribution of the path
given the bearing.

For models with a block structure and a prior that respects it, the posterior factorises further. Note that by
Lemma~\ref{lem:kron}(1), $L(0)=\hat c(T)^\T\otimes I_p$.

\begin{lem}\label{lem:ind mu0 Pim0} Assume Assumption~\ref{ass:kron} and $Z_0\sim N(\nu, \hat{A} \otimes I_p)$ with $\hat A\in\RR^{k\times k}$ positive definite. If $m_0=m(0,Z_0)$ and $\mu_0=\langle u, m_0\rangle$, then  $\mu_0$ and $\Pi m_0$ are independent. Moreover
\[ \mu_0\sim N(u^\T L(0)\nu, \gamma) \qquad \Pi m_0 \sim N(\Pi L(0) \nu, \gamma \Pi), \]
where $\gamma=\hat{c}(T)^\T \hat{A} \hat{c}(T)$ and $\Pi=I_p- uu^\T$. The distribution of $\Pi m_0$ is a degenerate normal distribution on $\RR^p$, supported on $u^\perp$.
\end{lem}
\begin{proof}
By the mixed-product identity for Kronecker products,  using the expression for $L(0)$ from Lemma \ref{lem:kron}(1),
\begin{align*}
	L(0) \mbox{Cov}(Z_0)L(0)^\T &=   L(0)(\hat A\otimes I_p)L(0)^\T\\ &
	  = (\hat c(T)^\top\otimes I_p)(\hat A\otimes I_p)(\hat c(T)\otimes I_p)
\\&  = \big(\hat c(T)^\top\hat A\,\hat c(T)\big)\otimes I_p = \gamma I_p.
\end{align*}
This implies
$L(0) Z_0 \sim N\left(L(0) \nu, \gamma I_p\right)$.
The result now follows upon noting

\[ \Bm \mu_0 \\ \Pi m_0\Em = \Bm u^\T \\ \Pi\Em L(0) Z_0 \sim N\left(\Bm u^\T L(0)\nu \\  \Pi L(0)\nu \Em, \gamma \Bm u^\T  u & u^\T \Pi \\ \Pi^\T u & \Pi \Pi^\T \Em\right). 	\]
By symmetry, $\Pi=\Pi^\T$ and $\Pi \Pi^\T = \Pi$. Since $u$ is a unit vector $u^\T u=1$. Finally $\Pi u = 0$.
We conclude $\mbox{Cov}(\mu_0, \Pi m_0)=0$ which, by joint normality, implies independence.
\end{proof}

\begin{cor}\label{cor:prior-block}
Under the assumptions of Lemma~\ref{lem:ind mu0 Pim0}, a sample from $Z_0$ conditional on the observation can be obtained by
\begin{enumerate}
	\item sampling $\mu_0$ from $p(\mu_0) \propto M_{p-1}(\mu_0, \sigma(0)) \phi(\mu_0; u^\T L(0)\nu, \gamma)$;
	\item sampling $\Pi m_0$ from $p(\Pi m_0) \propto \exp\left(-\frac1{2\sigma(0)^2} \|\Pi m_0\|^2\right) \phi(\Pi m_0; \Pi L(0)\nu, \gamma \Pi)$, a density with respect to Lebesgue measure on $u^\perp$;
	\item sampling $Z_0 \mid L(0)Z_0 = m_0$, where $m_0=\mu_0 u + \Pi m_0$.
\end{enumerate}
In the notation of Theorem~\ref{thm:prior}, $S=(\gamma+\sigma(0)^2)I_p$, so that $\xi^2=\gamma+\sigma(0)^2$ and
$\mu=u^\T L(0)\nu$.
\end{cor}
\begin{proof}
The density of the observation conditional on $Z_0=z_0$ is given by
\[ g(0,z_0) \propto M_{p-1}(\mu_0, \sigma(0))
\exp\left(-\frac1{2\sigma(0)^2} \|\Pi m_0\|^2\right),\]
by Corollary \ref{thm:gradlog_g}.
By Lemma \ref{lem:ind mu0 Pim0} the prior on $Z_0$ induces that $\mu_0$ and $\Pi m_0$ are independent and Gaussian; (i) and (ii) follow directly. $\Pi m_0$ and $\mu_0$ constrain $Z_0$ to $L(0)Z_0 = m_0$, which shows (3). The
expression for $S$ follows from $L(0)(\hat A\otimes I_p)L(0)^\T=\gamma I_p$ (see the proof of
Lemma~\ref{lem:ind mu0 Pim0}) and $\Sigma(0)=\sigma(0)^2I_p$.
\end{proof}

Here, $\mu_0$ and $\Pi m_0$ are independent a posteriori, and steps (ii) and (iii) are Gaussian draws. Step (i) is a
one-dimensional draw; it reduces to step (i) of Theorem~\ref{thm:prior} by introducing $R$ as an auxiliary variable.
Indeed, with $a=u^\T L(0)\nu$ and $\sigma=\sigma(0)$, the density of $\mu_0$ in step (1) is the marginal of
$(\mu_0,R)\mapsto R^{p-1}\phi(R;\mu_0,\sigma^2)\phi(\mu_0;a,\gamma)$ on $\RR\times(0,\infty)$. The marginal of $R$ is
proportional to $R^{p-1}\phi(R;a,\sigma^2+\gamma)$, which is step (i) of Theorem~\ref{thm:prior}, and
$\mu_0\mid R\sim N\big((\gamma R+\sigma^2a)/(\sigma^2+\gamma),\ \sigma^2\gamma/(\sigma^2+\gamma)\big)$.

\section{Sequential Monte Carlo for filtering}\label{sec:smc}

We now combine the results of the previous sections into a particle filter for the setting of the introduction. The latent process follows the model \eqref{eq:Z}, with $f$ satisfying Assumption~\ref{ass:dh}. The initial state has a Gaussian prior $Z_0\sim N(\nu,A)$. At times $0<t_1<\dots<t_{n_{\mathrm{obs}}}$ the directions $u_n=X_{t_n}/\|X_{t_n}\|$ are observed exactly. Let $\pi_n$ denote the filtering distribution of $Z_{t_n}$ given $u_{1:n}$, with $\pi_0=N(\nu,A)$. We wish to approximate $\pi_1,\dots,\pi_{n_{\mathrm{obs}}}$ and the likelihood $p(u_{1:n_{\mathrm{obs}}})=\prod_np(u_n\mid u_{1:n-1})$.

\subsection{One observation interval} By the Markov property, the interval $[t_{n-1},t_n]$ only involves the state $z$ at $t_{n-1}$ and the next direction $u_n$. Let $h_n(z)=h_{u_n}(t_{n-1},z)$ be the density of the direction of $X_{t_n}$ given $Z_{t_{n-1}}=z$ (Section~\ref{sec:nonlinear}, with $T=t_n$), and let $g_n(s,z)=g_{u_n}(s,z)$, $s\in[t_{n-1},t_n)$, be the guiding function of the linear model for the direction $u_n$ at time $t_n$. Then
$p(u_n\mid u_{1:n-1})=\int h_n(z)\,\pi_{n-1}(\dd z)$, 
but $h_n$ is intractable when $f\not\equiv0$. Theorem~\ref{prop:ac-nonlinear}, applied on $[t_{n-1},t_n]$, gives a way around this. Let $Z^\circ$ be the guided process \eqref{eq:guided-nonlinear} on $[t_{n-1},t_n]$ for the direction $u_n$, started at $Z^\circ_{t_{n-1}}=z$, with law $\PP^\circ_{u_n}$, and let
\begin{equation}\label{eq:Psi-n}
\Psi_n=\exp\Big(\int_{t_{n-1}}^{t_n}f(s,Z^\circ_s)^\T q^\T\nabla_z\log g_n(s,Z^\circ_s)\,\dd s\Big).
\end{equation}
The likelihood ratio \eqref{eq:weight-nonlinear} has expectation one under the law of the guided process, so that
\begin{equation}\label{eq:h-gPsi}
h_n(z)=g_n(t_{n-1},z)\,\PP^\circ_{u_n}\big[\Psi_n\big],
\end{equation}
and a path of the guided process with weight $\Psi_n$ is a weighted sample from the process conditioned on $u_n$. The guiding function $g_n(t_{n-1},z)$ is thus a tractable approximation of $h_n(z)$, and $\Psi_n$ corrects for its error. 

The guiding function need not come from the linear part $Bz$ of \eqref{eq:Z}. If $\tilde Bz$ is another linear drift with $(B-\tilde B)z$ in the range of $q$, then \eqref{eq:Z} is of the same form with $B$ replaced by $\tilde B$ and $f(t,z)$ by $f(t,z)+q^+(B-\tilde B)z$. The added term grows linearly in $z$, so that Theorem~\ref{prop:ac-nonlinear} applies by Corollary~\ref{cor:linear-growth} whenever $f$ satisfies Assumption~\ref{ass:dh}; a constant term in the linear drift is handled in the same way. As everything is conditional on $Z_{t_{n-1}}=z$, the linear model may moreover be chosen as a function of $z$. 

\subsection{Auxiliary particle filter} Identity \eqref{eq:h-gPsi} fits the auxiliary particle filter of \citet{PittShephard1999}, which we use in the formulation of \citet[Chapter~10]{Chopin2020introduction}. Suppose the particles $z^i$ with normalised weights $W^i$, $i=1,\dots,N$, approximate $\pi_{n-1}$. The guiding function serves as the auxiliary function, an approximation of the predictive density $h_n$ of the next direction, and twists the resampling: particle $i$ gets the first-stage weight $\lambda^i=W^ig_n(t_{n-1},z^i)$, and the ancestors are drawn with probabilities proportional to $\lambda^i$. Each resampled particle is then propagated by the guided process, which ends on the observed ray, and gets the second-stage weight $\Psi_n^i$. The particles with weights proportional to $\Psi_n^i$ approximate $\pi_n$, and
\[ \hat\ell_n=\Big(\sum_i\lambda^i\Big)\frac1N\sum_i\Psi_n^i \]
estimates $p(u_n\mid u_{1:n-1})$: given the particles at $t_{n-1}$, its expectation is $\sum_iW^ih_n(z^i)$ by \eqref{eq:h-gPsi}, and the product of the $\hat\ell_n$ is an unbiased estimate of the likelihood. Resampling before propagation spends the simulation of guided paths on particles that are compatible with $u_n$, and gives copies of a selected particle independent paths, so that they end at different points of the ray. For $f\equiv0$ the filter is fully adapted: the resampling uses the exact predictive density, the particles are propagated by the exact conditioned process, and $\Psi_n=1$. A bootstrap filter cannot be used with exact directions, as almost surely no particle propagated by the unconditioned dynamics lands on the observed ray.

\subsection{Exact first update}\label{sec:smc-first}
At the first observation the prior on the initial state is Gaussian, so that Theorem~\ref{thm:prior} can be used with the linear model. Instead of starting from draws from the prior, we draw $Z_0$ exactly from its posterior given $u_1$ under the linear model: the radius of $X_{t_1}=Ru_1$ is drawn from its one-dimensional density, and $Z_0$ from its Gaussian conditional distribution given $X_{t_1}=Ru_1$. The particles are then propagated by the guided process towards $u_1$ and weighted by $\Psi_1$, and the first likelihood increment is estimated by $\tilde p(u_1)$ times the average of $\Psi_1$, where $\tilde p(u_1)$ is the density of the first direction under the linear model, i.e.\ the normalising constant in Theorem~\ref{thm:prior}. This is again unbiased. The exact first update avoids reweighting the broad prior by an exact bearing in a single step.

\subsection{Algorithm}
Algorithm~\ref{alg:guided-filter} summarises the filter. For simplicity it discretises the guided process by the Euler scheme and $\log\Psi_n$ by a left-point Riemann sum, on a grid $t_{n-1}=s_0<s_1<\dots<s_M=t_n$. Any other scheme can be used; Section~\ref{sec:num-ncs} treats a stiff drift by splitting. If the linear model is chosen per particle, $g_n$ and $f$ in \textsc{GuidedProp} are those for the linear model of the ancestor. 

\begin{algorithm}[!t]
\caption{Guided auxiliary particle filter for the model \eqref{eq:Z} with prior $Z_0\sim N(\nu,A)$, exact directions $u_1,\dots,u_{n_{\mathrm{obs}}}$ observed at times $t_1<\dots<t_{n_{\mathrm{obs}}}$, and $N$ particles.}
\label{alg:guided-filter}
\begin{algorithmic}[1]
\Procedure{GuidedProp}{$z,n$} \Comment{guided process on $[t_{n-1},t_n]$ towards $u_n$}
  \State $\log\Psi\gets0$
  \For{$j=0,\dots,M-1$}
    \State $\delta\gets s_{j+1}-s_j$; \quad $r\gets\nabla_z\log g_n(s_j,z)$ \Comment{Theorem~\ref{thm:general case g and nabla log g}}
    \State $\log\Psi\gets\log\Psi+f(s_j,z)^\T q^\T r\,\delta$
    \State $z\gets z+\big[Bz+qf(s_j,z)+Qr\big]\delta+q\sqrt{\delta}\,\zeta$, \quad $\zeta\sim N(0,I_m)$
  \EndFor
  \State \Return $(z,\Psi)$
\EndProcedure
\Statex
\Statex \textbf{First observation} (Section~\ref{sec:smc-first})
\State compute $\bar m$, $S$, $\xi$, $\mu$ of Theorem~\ref{thm:prior} for the direction $u_1$ and the linear model on $[0,t_1]$, and $\tilde p(u_1)=\int_0^\infty R^{p-1}\phi(Ru_1;\bar m,S)\dd R$
\For{$i=1,\dots,N$}
  \State draw $z_0^i$ from the posterior of $Z_0$ given $u_1$ \Comment{Theorem~\ref{thm:prior}}
  \State $(z^i,\Psi^i)\gets\Call{GuidedProp}{z_0^i,1}$
\EndFor
\State $\hat\ell_1\gets\tilde p(u_1)\,\frac1N\sum_{i}\Psi^i$; \quad $W^i\gets\Psi^i/\sum_{k}\Psi^k$
\Statex
\Statex \textbf{Further observations}
\For{$n=2,\dots,n_{\mathrm{obs}}$}
  \State $\lambda^i\gets W^i\,g_n(t_{n-1},z^i)$, $i=1,\dots,N$ \Comment{first-stage weights}
  \State draw ancestors $a^1,\dots,a^N$ by systematic resampling with weights proportional to $\lambda^1,\dots,\lambda^N$
  \For{$i=1,\dots,N$}
     \State $(\hat z^i,\Psi^i)\gets\Call{GuidedProp}{z^{a^i},n}$ \Comment{second-stage weights}
  \EndFor
  \State $z^i\gets\hat z^i$; \quad $\hat\ell_n\gets\big(\sum_{i}\lambda^i\big)\frac1N\sum_{i}\Psi^i$; \quad $W^i\gets\Psi^i/\sum_{k}\Psi^k$
\EndFor
\State \Return the log-likelihood estimate $\sum_{n}\log\hat\ell_n$, and the weighted particles $(z^i,W^i)$ approximating $\pi_n$, $n=1,\dots,n_{\mathrm{obs}}$
\end{algorithmic}
\end{algorithm}

\section{Numerical illustrations}\label{sec:numerics}

\subsection{The discretised conditioned process ends on the ray}\label{sec:num-paths}

We simulate the conditioned process $Z^\star$ of \eqref{eq:Zcirc} for the nearly-constant-velocity model
(Example \ref{ex:Nearly-constant-velocity}, $p=2$) with  $T=1$, $x_0=(1,0)$, $v_0=(0,1.5)$ and observed angle
$y_T=-\pi/5$. For discretisation we use the implicit scheme from Lemma \ref{lem:impl_euler_ncv} with 
$\delta=T/M$.
Figure \ref{fig:paths}(a) shows $25$ sample paths for $M=400$. All paths end on the ray, in agreement with
Theorem \ref{lem:strong-nonlinear}(2) with $f\equiv0$; numerically, the distance $\|\Pi X_T\|$ to the ray at time $T$, averaged over the paths, is $2\cdot10^{-17}$, i.e.\ zero up to rounding.

With $f\equiv0$, the law of the conditioned process is the conditional law given $U=u$ (Theorem \ref{lem:strong-nonlinear}(2)). Hence, writing $X_T=Ru$ in polar coordinates and using $\Sigma(0)=\si(0)^2I_p$, the terminal radius $R=\langle X_T^\star,u\rangle$ has density proportional to $r\,\phi(r;\mu(0,z_0),\si(0)^2)$ on $(0,\infty)$. Figure \ref{fig:paths}(b) compares a histogram of $R$ from $2\cdot 10^4$ discretised paths ($M=200$) with this density. The agreement is excellent. 

\begin{figure}[t]
\centering
\includegraphics[width=\textwidth]{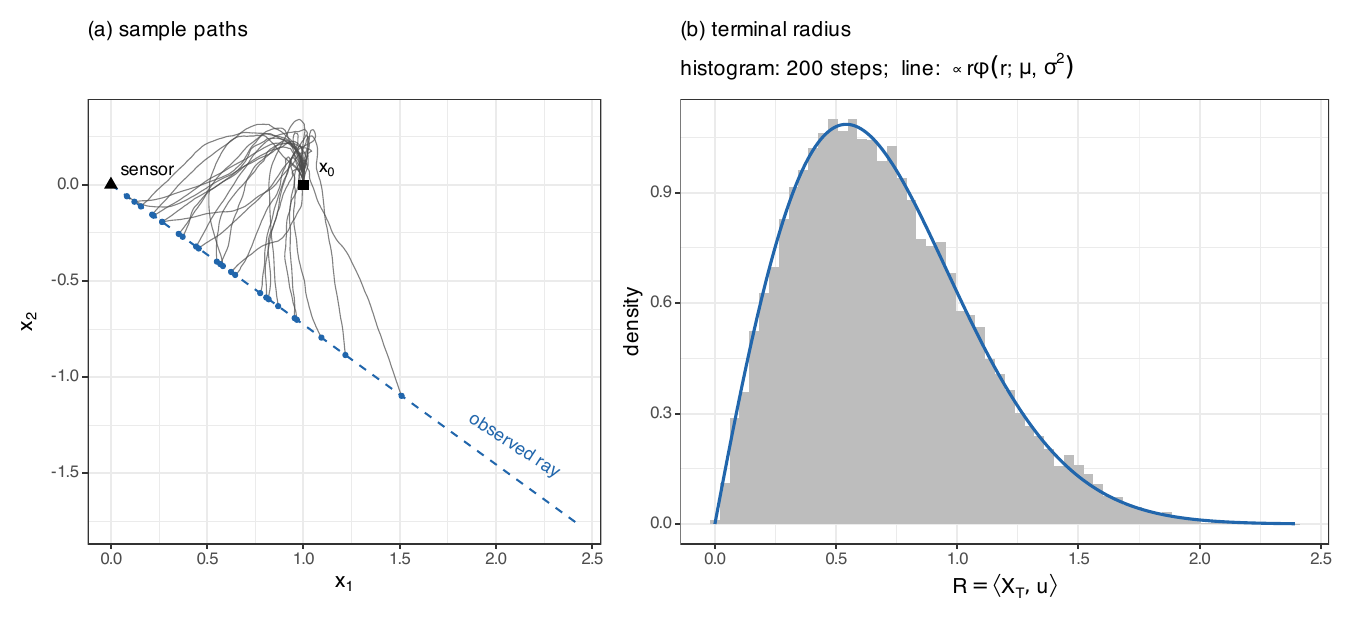}
\caption{Conditioned nearly-constant-velocity process obtained by discretising \eqref{eq:Zcirc}. (a) $25$ paths; dots are the endpoints. (b) Histogram of the terminal radius $\langle X_T,u\rangle$ from $2\cdot10^4$ discretised paths, and the exact density $\propto r\,\phi(r;\mu(0,z_0),\si(0)^2)$.}
\label{fig:paths}
\end{figure}

\subsection{A nonlinear drift: nearly-constant-speed target}\label{sec:num-ncs}

 As an example of practical relevance we consider a target that moves at a nearly constant speed while its heading changes randomly. Such a model is natural for ships and aircraft in bearings-only target motion analysis:
\begin{equation}\label{eq:ncs}
\dd X_t=V_t\dd t,\qquad \dd V_t=b(V_t)\dd t+\eta\dd W_t,\qquad b(v)=-\ka\left(\|v\|^2-v_*^2\right)v,
\end{equation}
with $X_t,V_t\in\RR^2$. The drift pulls the speed $\|V_t\|$ towards $v_*$ at rate $2\ka v_*^2$ and leaves the heading unaffected. 

 For \eqref{eq:ncs} the $h$-function is intractable. Let $\tilde g$ be the guiding function (Definition \ref{def:g}) of a linear auxiliary model
\[ \dd X_t=V_t\dd t,\qquad \dd V_t=(\tilde JV_t+\tilde\beta)\dd t+\eta\dd W_t, \]
with $\tilde J\in\RR^{2\times2}$ symmetric and $\tilde\beta\in\RR^2$. This auxiliary model does not satisfy Assumption \ref{ass:kron} (block structure) unless $\tilde J$ is a multiple of the identity. Still, $X_T\mid Z_t=z$ is Gaussian with explicitly computable mean and (in general non-isotropic) covariance: $\tilde g$ and $\tilde r=\nabla_z\log\tilde g$ are given in closed form by Theorem \ref{thm:general case g and nabla log g}. The guided process has velocity evolving according to the SDE
\[ \dd V^\circ_t=\left[b(V^\circ_t)+\eta^2\tilde r_V(t,Z^\circ_t)\right]\dd t+\eta\dd W^\circ_t, \]
where $\tilde r_V$ denotes the velocity block of $\tilde r$; $\PP^\circ$ denotes the law of $Z^\circ$ started at $Z^\circ_0=z$. We cannot apply Theorem \ref{prop:ac-nonlinear} as $b$ is not bounded. To resolve this, we fix a large constant $K$ and redefine $b$ to $b(v)=\bar{b}(P_K v)$, where $\bar{b}(v)=-\ka\left(\|v\|^2-v_*^2\right)v$ and $P_K v= v \min(1,K/\|v\|)$. Thus, we evaluate the original drift after first  projecting $v$ onto a ball of radius $K$. A large value of $K$ will make the projection step irrelevant for numerical simulations. With this redefinition $b$ is bounded, and the auxiliary drift $\tilde Jv+\tilde\beta$ adds a term of linear growth, so that Theorem \ref{prop:ac-nonlinear} applies by Corollary \ref{cor:linear-growth} and gives
\begin{equation}\label{eq:girsanov}
h(0,z)=\tilde g(0,z)\,\PP^\circ\left[\Psi\right],\qquad \log\Psi=\int_0^T\left(b(V^\circ_s)-\tilde JV^\circ_s-\tilde\beta\right)^\T\tilde r_V(s,Z^\circ_s)\dd s.
\end{equation}
 We consider three auxiliary models:
\begin{itemize}
\item the \emph{NCV auxiliary} (NCV=``Nearly Constant Velocity''), $\tilde J=0$, $\tilde\beta=0$, for which $\tilde g$ is given by Corollary \ref{thm:gradlog_g};
\item the \emph{speed-circle auxiliary}: the linearisation of $b$ at $\bar v=v_*v_0/\|v_0\|$, the point on the circle $\{v \colon \|v\|=v_*\}$ in the current heading, where $v_0$ is the velocity at the start of the interval. This gives $\tilde J=-2\ka v_*^2\,ee^\T$ with $e=v_0/\|v_0\|$, and $\tilde\beta=-\tilde J\bar v$: the speed along the current heading relaxes to $v_*$, and there is no drift orthogonal to $e$; 
\item the \emph{Taylor auxiliary}: the first-order Taylor expansion of $b$ at $v_0$ itself, $\tilde J=Db(v_0)=-\ka\big[(\|v_0\|^2-v_*^2)I+2v_0v_0^\T\big]$ and $\tilde\beta=b(v_0)-\tilde Jv_0$. 
\end{itemize}
If $\kappa$ is large, mean reversion of the speed to $v_*$ is strong, so that the speed is close to $v_*$ during most of each observation interval. In that case the speed-circle auxiliary, which linearises $b$ on the speed circle, is likely to perform better than the Taylor auxiliary, which linearises $b$ at the initial speed: its drift is closer to $b$, so that the guided process is closer to the true conditioned process and the weights $\Psi$ vary less.

\subsubsection{Discretisation.} The guided process is simulated on the grid $t_j=h(1-(1-j/M)^2)$ with $M=200$ steps per interval of length $h$. As in Section \ref{sec:discr_ncv}, we simulate the processes by simulation of the angular and radial processes, which are dependent unless $\kappa=0$.  As the flow of $\dot v=b(v)$ can be  solved exactly,
\begin{equation}
	\label{eq:exactflow}
	\Phi_t(v(0)):= v(t) = v(0)\left(\frac{\alpha_0}{1-(1-\alpha_0)e^{-2\kappa v_*^2 t}}\right)^{1/2}, \quad \text{with} \quad \alpha_0 = \frac{v_*^2}{\|v(0)\|^2}, 
\end{equation} 
we use a splitting scheme. The guiding term and the noise are added by an explicit and implicit Euler step for the radial and angular parts respectively. Details are given in Section \ref{sec:discretisation_nsc_model} of the appendix. The integral in \eqref{eq:girsanov} is approximated by a left-point Riemann sum.

\subsubsection{Validation.} Table \ref{tab:ncs-h} compares, over a single interval, $\tilde g(0,z)\PP^\circ[\Psi]$ with a brute-force estimate of $h(0,z)$. Here $h(0,z)$ is the density of the angle of $X_h$ at a fixed $\psi$. The brute-force estimate uses a histogram of $10^6$ forward simulations of \eqref{eq:ncs}. The setting is $\ka=400$, $v_*=0.05$, $\eta=0.01$, $h=1$ and $x_0=(-0.05,0.7)$. When $v_0$ is close to the speed circle all auxiliary models work well. When $v_0=(0.02,-0.06)$ is not, the NCV guiding function alone is off by $25\%$ ($35.6$ versus $47.1$), and the Taylor guiding function by $11\%$ in the other direction ($52.5$), but the correction $\PP^\circ[\Psi]$ removes this bias in both cases. The speed-circle auxiliary gives $\tilde g$ close to $h$, and weights with $\mathrm{ESS}/n=0.99$, against $0.81$ and $0.73$ for the NCV and Taylor auxiliaries.

\begin{table}[t]
\centering
\begin{tabular}{llrrr}
\toprule
$v_0$ & method & $\tilde g$ & estimate of $h$ & ESS$/n$ \\
\midrule
$(0.000, -0.051)$ & brute force & -- & $44.0 \pm 0.22$ & -- \\
$(0.000, -0.051)$ & NCV auxiliary & 43.2 & $43.7 \pm 0.06$ & 0.97 \\
$(0.000, -0.051)$ & linearised at velocity & 43.8 & $43.7 \pm 0.03$ & 0.99 \\
$(0.000, -0.051)$ & speed-circle auxiliary & 43.2 & $43.7 \pm 0.03$ & 0.99 \\
\midrule
$(0.020, -0.060)$ & brute force & -- & $47.1 \pm 0.24$ & -- \\
$(0.020, -0.060)$ & NCV auxiliary & 35.6 & $47.4 \pm 0.16$ & 0.81 \\
$(0.020, -0.060)$ & linearised at velocity & 52.5 & $47.5 \pm 0.20$ & 0.73 \\
$(0.020, -0.060)$ & speed-circle auxiliary & 43.6 & $47.1 \pm 0.04$ & 0.99 \\
\bottomrule
\end{tabular}
\caption{Density $h(0,z)$ of the direction of $X_1$ at a fixed angle for the nearly-constant-speed model ($\ka=400$), for two initial velocities. The brute-force estimate uses $10^6$ forward simulations; the guided estimates $\tilde g(0,z)\PP^\circ[\Psi]$ use $2\cdot10^4$ guided paths, simulated with the implicit scheme with $M=200$. $\mathrm{ESS}/n$ is the relative effective sample size of the weights $\Psi$.}
\label{tab:ncs-h}
\end{table}

\subsubsection{Filtering application.} We use the classical bearings-only scenario of \citet{gordon1993novel}, formulated in continuous time, with dynamics \eqref{eq:ncs}.
\begin{itemize}
\item \emph{Model}: \eqref{eq:ncs} with $v_*=0.05$, $\ka\in\{100,400,1600\}$ and $\eta\in\{0.01,0.02\}$, observed at times $1,2,\dots,24$. For each $(\ka,\eta)$ the true path is simulated with the same driving noise.
\item \emph{True initial state}: $x_0=[-0.05, 0.7]$, $v_0=[0.001, -0.055]$, as in \citet{gordon1993novel}.
\item \emph{Prior}: we impose independent priors on the initial position and velocity where 	\[ X_0 \sim N_2\left([0, 0.4]^\T, \mbox{diag}(0.5^2, 0.3^2)\right) \quad \text{and}\quad  V_0\sim N_2\left([0, 0]^\T, 0.05^2I_2\right).\] The prior on the initial position is mildly informative. The prior on the initial velocity is centred at zero, with a spread comparable to the regulated speed $v_*$ in every direction, so that it carries hardly any information on the heading.
\item \emph{Observations}: the exact bearings $u_n=X_n/\|X_n\|$.
\end{itemize}
Figure \ref{fig:ncs-data} shows the observed bearings together with a visualisation of the prior on the initial position. The true initial state and initial velocity vector of the data are displayed in black. 

\begin{figure}[t]
\centering
\includegraphics[width=0.75\textwidth]{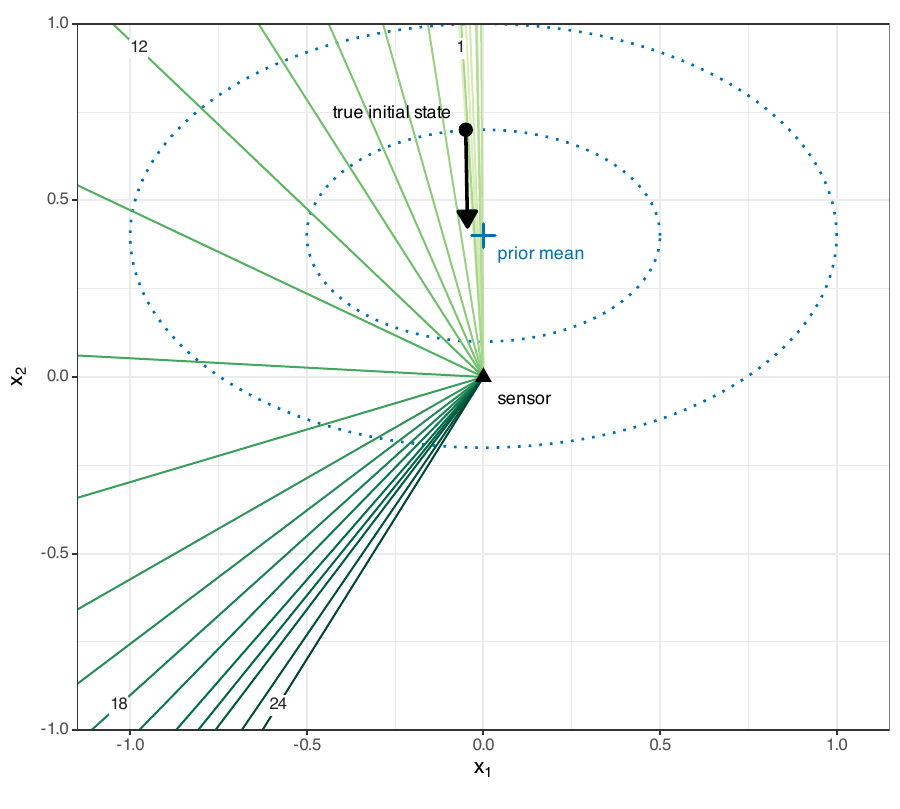}
\caption{Data and prior for the nearly-constant-speed target ($\ka=400$, $\eta=0.01$). Observed bearings from the sensor at the origin, shaded from light green (first observation) to dark green (last observation), some labelled with their observation number; prior on the initial position (mean and dotted $1$- and $2$-standard-deviation ellipses); true initial position and velocity (the arrow shows the displacement over $5$ units of time). The prior on the initial velocity, centred at zero, is not shown.}
\label{fig:ncs-data}
\end{figure}
We call $\|X_t\|$, the distance between the target and the sensor at the origin, the \emph{range}. As the bearings are exact, the direction $X_{t_n}/\|X_{t_n}\|=u_n$ is known at each observation time, so that the range is the only unknown component of the position there. 

We apply Algorithm~\ref{alg:guided-filter}. In the notation of \eqref{eq:Z}, $Z=(X,V)$, the linear part is the NCV model, $q=[0,\eta I_2]^\T$ and $f=b/\eta$, so that $\Psi_n$ in \eqref{eq:Psi-n} is $\Psi$ in \eqref{eq:girsanov} with the NCV auxiliary. At the first observation all filters use the exact first update with the NCV auxiliary. From the second observation onwards, each filter guides with its own auxiliary model, chosen per particle from its velocity at $t_{n-1}$, and $\Psi$ is that of \eqref{eq:girsanov}. 

\subsubsection{Results for filtering application.}

We simulated a trajectory of the nearly constant speed model with $\kappa=400$ and $\eta=0.01$. We simulated the data for $t\in[0,24]$ by the splitting scheme (exact flow of $b$, then an Euler step for the noise) with step size $10^{-3}$. From the simulated trajectory we retained as observations the bearings at times $1,2,\ldots, 24$. We apply all three guided particle filters using $N=5000$ particles and $M=200$ for the discretisation of the guided process.

In Figure \ref{fig:ncs-scenario} we show the results for the three guided particle filters. We show the mean of the filtering distribution and bars with endpoints determined by the $10\%$ and $90\%$ quantiles of the filtering distribution of the range. The three filters give nearly identical results, with the Taylor auxiliary performing slightly better. One can see that uncertainty is large at about the first $7$ time instances and decreases afterwards; the true range lies outside the band only at observations $5$ and $11$, and for the speed-circle auxiliary also at observation $4$.  We also ran the filter  without specific handling of the first state, initialising particles from the prior, and observed that in that case  the three filters give different, worse, results.

Since the prior on the initial velocity carries hardly any information on the heading, the range is poorly determined during the first observations, until the bearing starts to turn. Figure \ref{fig:ncs-clouds} shows the particles of single runs with $N=5000$ particles at observations $4$, $8$, $12$ and $15$. As the bearings are exact, the propagated particles lie on the observed rays, so that the filtering distribution of the position at an observation time is that of the range along the ray. The three filters give practically the same, broad filtering distribution. The exact first update is essential here. When the filters start from draws from the prior instead, a single run often represents the broad filtering distribution of the early phase by a few narrow clusters, whose location varies from run to run and between the filters.

\subsubsection{Comparison with shifted Rayleigh filter.}
We compare with the shifted Rayleigh filter from  \citep{Clark2007SRF} which approximates the filtering distribution by a Gaussian. For exact bearings its update is the exact conditioning of the Gaussian prediction on the observed direction, as in Theorem~\ref{thm:prior}, followed by moment matching. As the filter needs linear dynamics, the prediction uses the NCV model, or the speed-circle or Taylor linearisation of $b$ at the filtered mean velocity.

 Figure~\ref{fig:ncs-srf} shows the result for the data of Figure~\ref{fig:ncs-scenario}. In the early phase all three variants place the target too far from the sensor, with errors of up to $0.45$. Their $10$--$90\%$ bands no longer contain the true range from observation $9$ (speed-circle), $10$ (Taylor) or $13$ (NCV) onwards, and those of the speed-circle and Taylor predictions already miss it at most of the observations $4$--$7$. With the NCV prediction the range is not observable, and the filtered range drifts away from the truth. With the speed-circle and Taylor predictions the filter follows the target from observation $15$ onwards, but with errors of $0.02$--$0.04$ (Figure~\ref{fig:ncs-srf}(c)), about ten times larger than those of the guided filters in Figure~\ref{fig:ncs-scenario}(c), and with bands that do not contain the true range at any of the observations $15$--$24$. A Gaussian can represent neither the broad, skewed filtering distribution of the range in the early phase nor the dependence between position and velocity induced by the speed regulation. The shifted Rayleigh filter is very fast (about $0.3$\,ms per run of $24$ observations, against $3$--$5$\,s for the guided filters with $N=5000$ and $M=200$ on a single core), but this speed comes at the cost of the errors described above.

\subsubsection{A small study for comparing guided and shifted Rayleigh filters.}
To compare the two filters beyond a single data set, we simulated $30$ data sets for each $\ka\in\{250,500,1000\}$, with $\eta=0.01$. For each data set the true initial position and velocity are drawn from the prior, and the same driving noise and initial state are used for the three values of $\ka$. Both filters use the speed-circle approximation: the guided filter ($N=5000$) as its auxiliary model, the shifted Rayleigh filter for its prediction. As the exact filtering distribution and likelihood are not available, we use the guided filter with the speed-circle auxiliary and $N=10^5$ particles as a reference. The results are summarised in Figure~\ref{fig:ncs-study}.
\begin{itemize}
\item The filtered mean range of the guided filter is $4$--$9$ times closer to that of the reference than the one of the shifted Rayleigh filter (median distance $0.007$--$0.008$, against $0.035$--$0.056$).
\item The log-likelihood of the shifted Rayleigh filter is biased downwards, with median error $-4.7$, $-5.4$ and $-8.8$ for $\ka=250$, $500$ and $1000$, and errors of $30$--$120$ for targets that pass close to the sensor. For the guided filter the median absolute error is $0.24$--$0.96$, and all errors lie between $-9.1$ and $4.4$.
\item As the data are drawn from the prior, the $10$--$90\%$ bands of the exact filter contain the true range with probability $0.8$. The fraction of observations at which the band contains the true range is $0.76$--$0.78$ for the reference, $0.70$--$0.77$ for the guided filter and $0.58$--$0.60$ for the shifted Rayleigh filter, whose bands are too narrow. From observation $13$ onwards the RMSE of the filtered range is about that of the reference for the guided filter, and $2$--$3$ times larger for the shifted Rayleigh filter.
\end{itemize}

\subsubsection{Moving sensor.}
Our approach carries through, with minor modifications, when the sensor is moving. We provide details in Appendix~\ref{sec:moving-sensor}.

\subsubsection{Estimating the loglikelihood.}
Figure \ref{fig:ncs-perf} shows the effective sample size and the RMSE of the filtered position over $20$ runs for each of the six combinations of $\ka$ and $\eta$, with $N=5000$ and $N=10\,000$ particles. The log-likelihood estimates are shown in Figure \ref{fig:ncs-loglik}, and below we summarise their spread by the interquartile range; we call a run collapsed if its log-likelihood estimate is more than $20$ below the median. 
\begin{itemize}
\item The guided filters are stable in all settings. Only $1$ of the $720$ runs collapses: a run of the Taylor auxiliary with $N=10\,000$ for $\ka=1600$ and $\eta=0.01$. For $\ka\le400$ the interquartile range of the log-likelihood estimates is at most $0.61$ with $N=5000$ and at most $0.52$ with $N=10\,000$; for $\ka=1600$ it is between $0.78$ and $1.8$.
\item The effective sample size decreases with $\ka$ and $\eta$, from $\mathrm{ESS}/N=0.42$--$0.55$ at $\ka=100$ to $0.20$--$0.36$ at $\ka=1600$, $\eta=0.01$ and $0.10$--$0.19$ at $\ka=1600$, $\eta=0.02$. A larger $\eta$ makes the path between observations less predictable, so that the weights $\Psi$ vary more.
\item The auxiliary models differ mainly in the effective sample size. For $\ka=100$ the Taylor auxiliary is best ($\mathrm{ESS}/N\approx0.52$--$0.55$ with $N=5000$). For $\ka\ge400$ the speed-circle auxiliary has the largest effective sample size, and for $\ka=1600$, $\eta=0.02$ the Taylor auxiliary falls back to the level of the NCV auxiliary.
\item The median RMSE of the filtered position is between $0.11$ and $0.14$ in all settings. It is dominated by the first observations, where the range is poorly determined (Figure \ref{fig:ncs-scenario}), and hardly distinguishes the filters.
\end{itemize}

\emph{Reference values and time discretisation.} As the likelihood of the exact bearings is not available in closed form, we use as a reference the median of $5$ runs of the guided filter with the speed-circle auxiliary and $N=10^6$ particles (dashed lines in Figure~\ref{fig:ncs-loglik}). For $\ka\le400$ the three filters agree with this reference to within $0.3$. For $\ka=1600$ they differ by up to about $1.5$. As each filter is unbiased for the likelihood of its own discretised guided process, this reflects the time discretisation: refining the grid from $M=200$ to $M=800$ steps per interval for $\ka=1600$ and $\eta=0.02$ shows that with $M=200$ the log-likelihood is overestimated by about $0.5$ with the speed-circle auxiliary, $0.1$ with the Taylor auxiliary and $1.1$ with the NCV auxiliary, whose drift differs most from $b$. A common part of this bias arises in the first interval, where particles whose initial speed is far from $v_*$ relax to it within the first time step of the grid. These biases do not affect the comparison of the variability of the estimates above, but for strong speed regulation a finer grid or the speed-circle auxiliary should be used when accurate likelihood values are needed.

\begin{figure}[t]
\centering
\includegraphics[width=\textwidth]{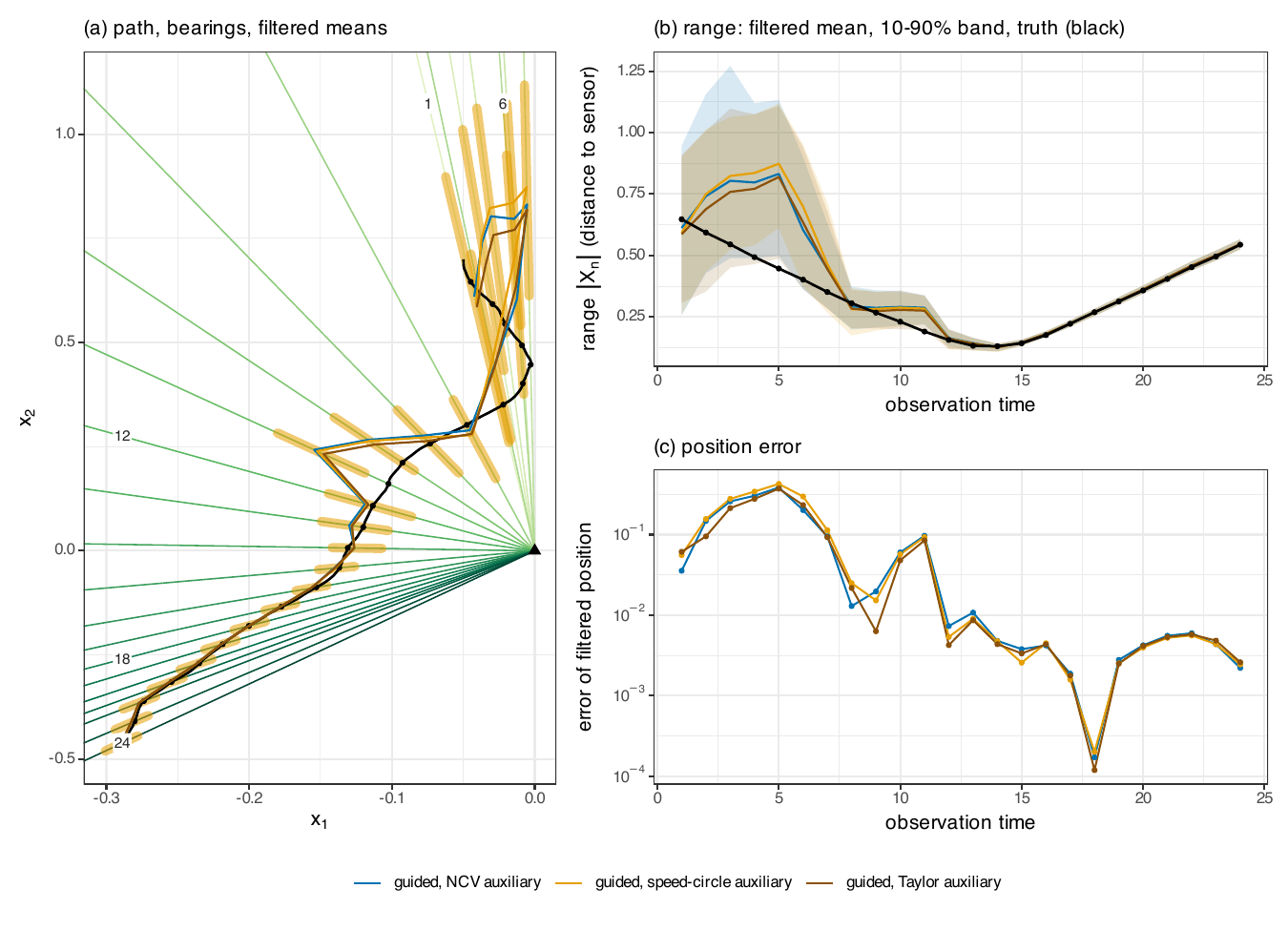}
\caption{Nearly-constant-speed target with exact bearings ($\ka=400$, $\eta=0.01$, $N=5000$). (a) True path (black, dots at observation times), observed bearings and filtered means; the shaded segments along the bearings are the intervals between the $10\%$ and $90\%$ quantiles of the filtering distribution of the range for the speed-circle auxiliary (the other filters give practically the same intervals); the bearings are shaded from light green (first observation) to dark green (last observation), and some are labelled with their observation number.  (b) Filtered mean of the range $\|X_{t_n}\|$ (the distance to the sensor) with the band between its $10\%$ and $90\%$ quantiles, and the true range (black). (c) Distance between filtered mean and true position.}
\label{fig:ncs-scenario}
\end{figure}

\begin{figure}[t]
\centering
\includegraphics[width=\textwidth]{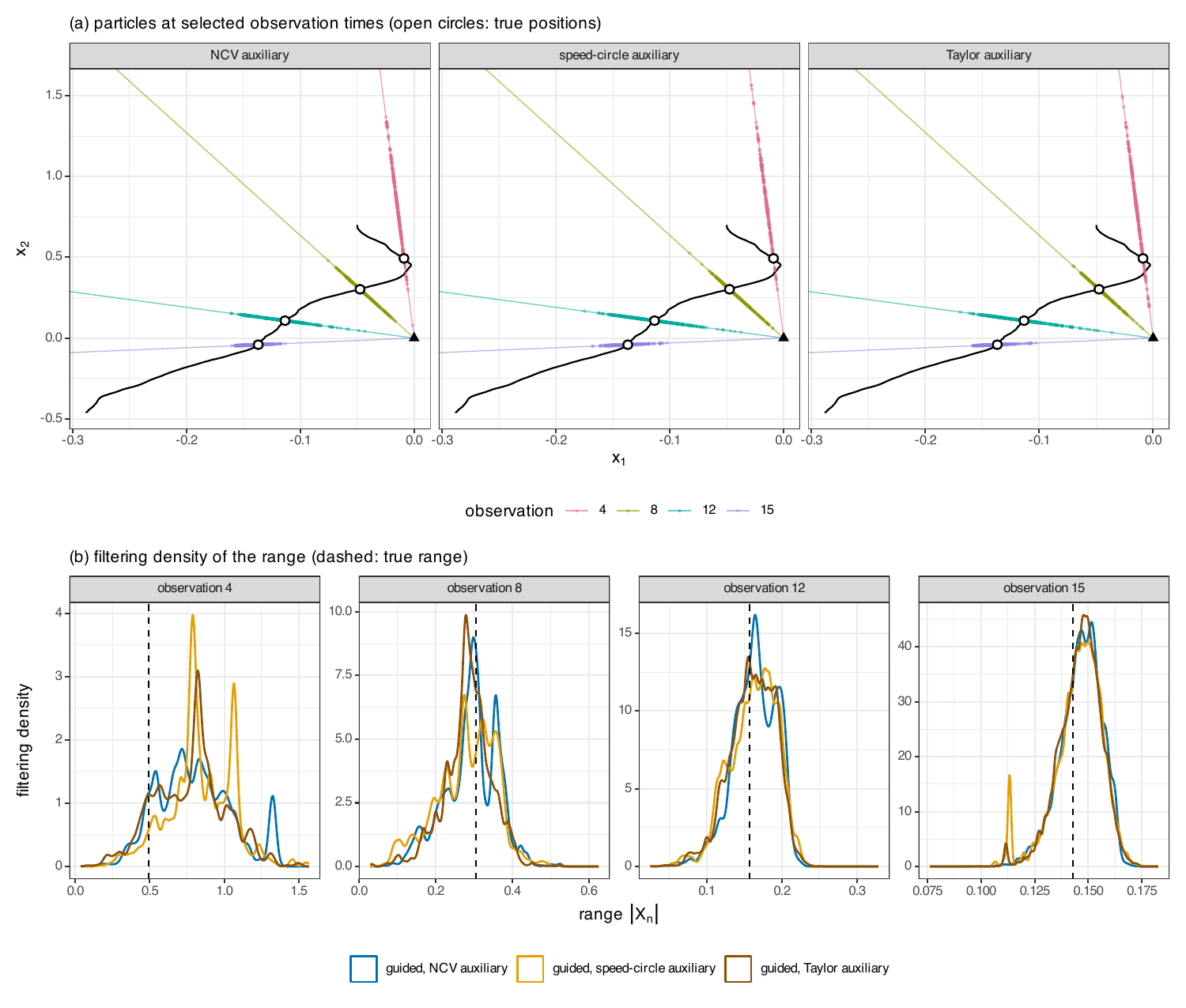}
\caption{Particles of single runs of the three guided filters with $N=5000$ particles ($\ka=400$, $\eta=0.01$, same data as in Figure \ref{fig:ncs-scenario}). (a) Particle positions at observations $4$, $8$, $12$ and $15$ (an equally weighted subsample of $400$ particles), the corresponding bearings, the true path and the true positions at these times (open circles). (b) Weighted density of the range $\|X_n\|$ of the particles, with the true range (dashed).}
\label{fig:ncs-clouds}
\end{figure}

\begin{figure}[t]
\centering
\includegraphics[width=\textwidth]{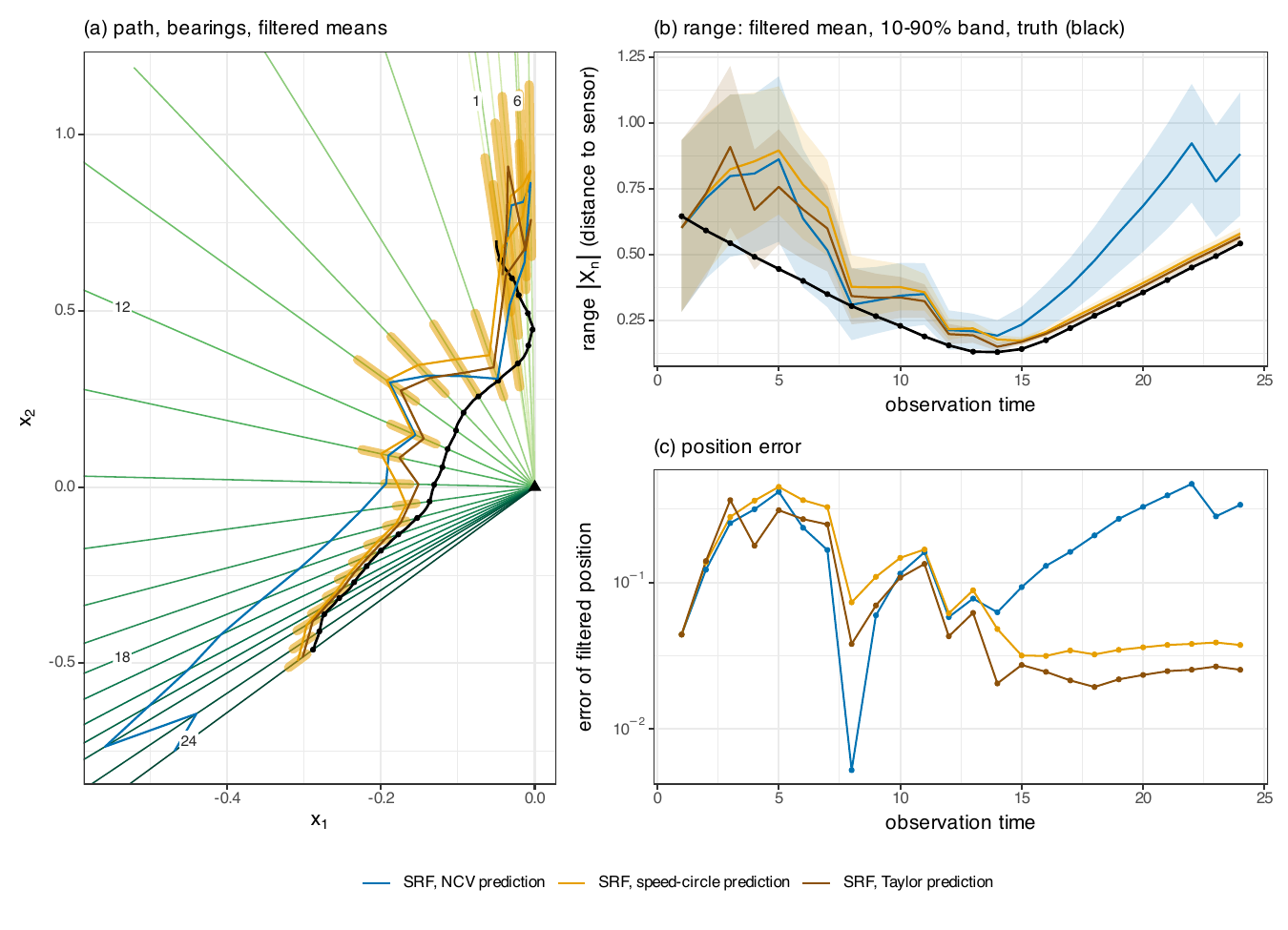}
\caption{As Figure~\ref{fig:ncs-scenario}, for the shifted Rayleigh filter with three linear models for the prediction: the NCV model, and the speed-circle and Taylor linearisations of $b$ at the filtered mean velocity. In (a), the shaded segments are the $10$--$90\%$ intervals of the range for the speed-circle prediction. Note that subfigure (a) has extended axes compared to Figure~\ref{fig:ncs-scenario}.}
\label{fig:ncs-srf}
\end{figure}

\begin{figure}[t]
\centering
\includegraphics[width=0.9\textwidth]{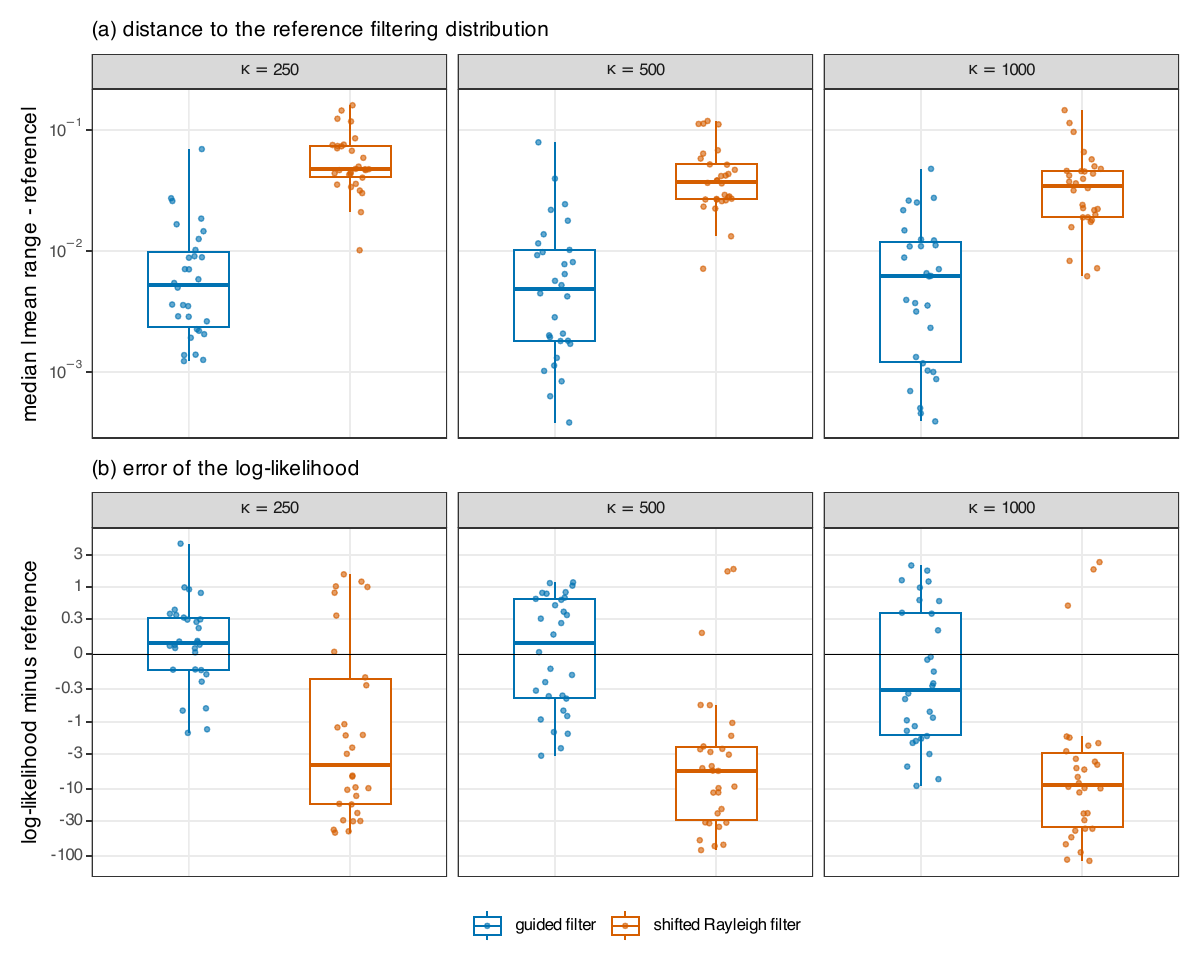}
\caption{Simulation study comparing the guided filter ($N=5000$) and the shifted Rayleigh filter, both with the speed-circle approximation, for $\eta=0.01$ and $\ka\in\{250,500,1000\}$; $30$ data sets per $\ka$ with the initial state drawn from the prior, one dot per data set. (a) Median over the $24$ observations of the distance between the filtered mean range and that of the reference (guided filter with $N=10^5$ particles). (b) Log-likelihood estimate minus that of the reference (pseudo-logarithmic axis).}
\label{fig:ncs-study}
\end{figure}

\begin{figure}[t]
\centering
\includegraphics[width=0.85\textwidth]{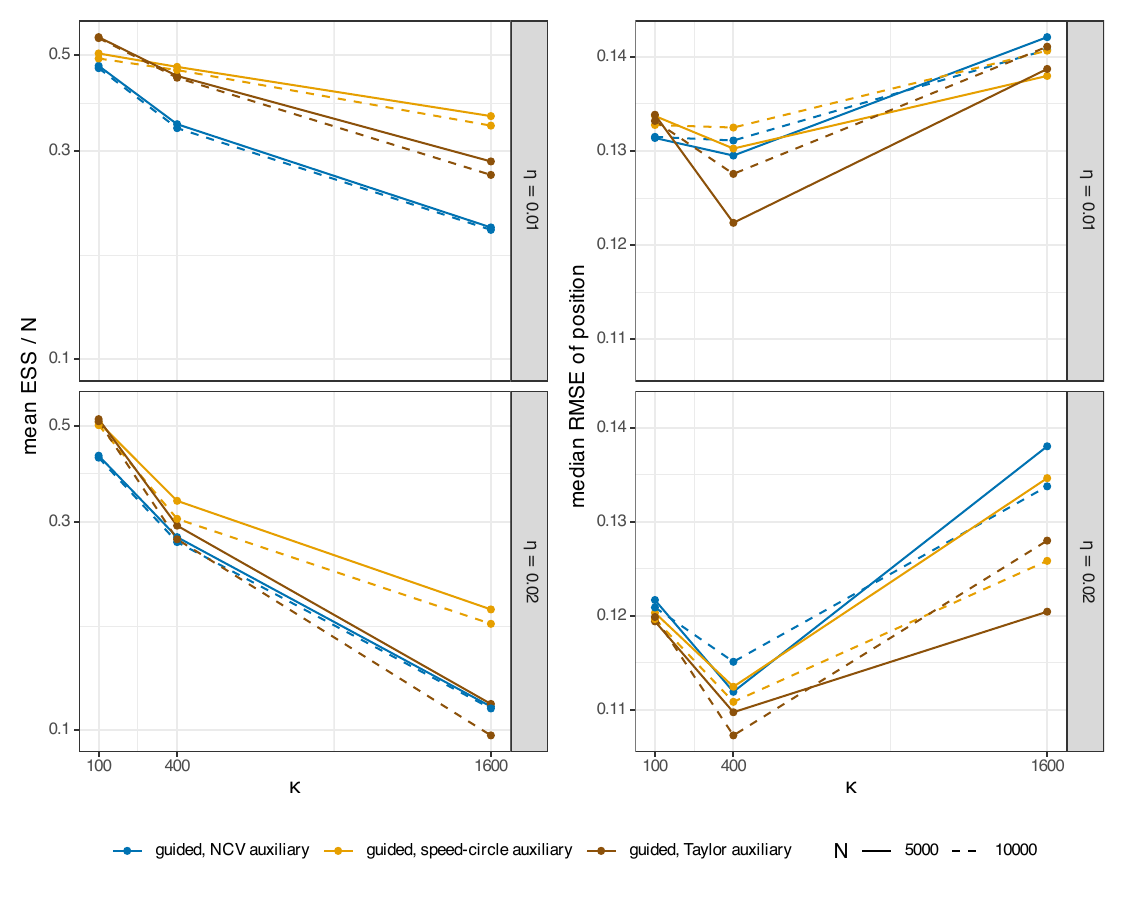}
\caption{Nearly-constant-speed target with exact bearings, $20$ runs for each $\ka\in\{100,400,1600\}$ and $\eta=0.01$ (top) or $\eta=0.02$ (bottom), with $N=5000$ (solid) and $N=10\,000$ (dashed) particles, for the three auxiliary models. Left: time-averaged $\mathrm{ESS}/N$. Right: median RMSE of the filtered position.}
\label{fig:ncs-perf}
\end{figure}

\begin{figure}[t]
\centering
\includegraphics[width=\textwidth]{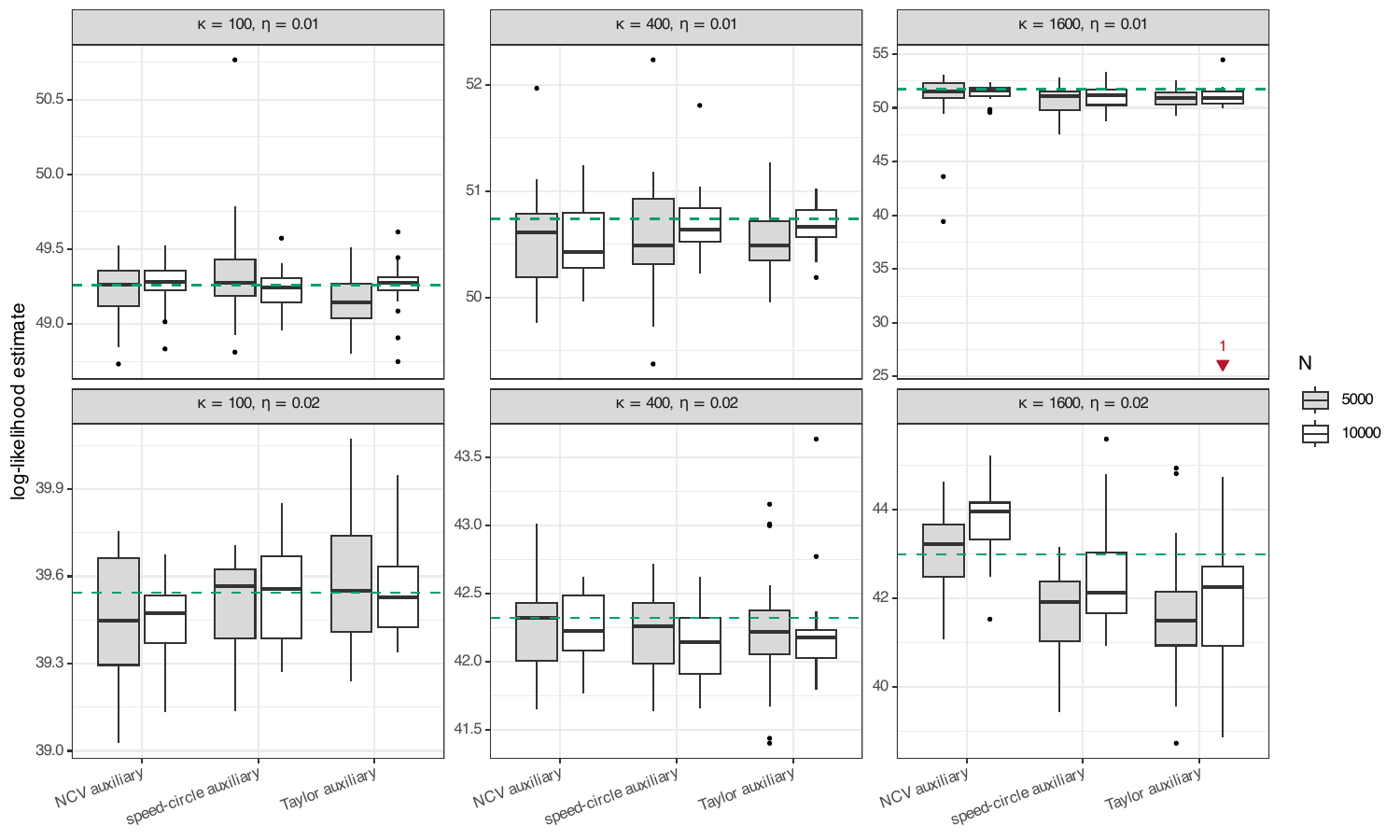}
\caption{Log-likelihood estimates of the guided filters for the three auxiliary models, $20$ runs for each $(\ka,\eta)$, with $N=5000$ (grey) and $N=10\,000$ (white) particles. In each panel the vertical axis is cut $25$ below the median of all estimates in the panel; estimates below the cut (collapsed runs) are drawn as triangles at the lower edge, with their number. The dashed line is a reference value: the median of $5$ runs with the speed-circle auxiliary and $N=10^6$ particles.}
\label{fig:ncs-loglik}
\end{figure}

\appendix

\section{Some results on $M_n$}

\begin{lem}[Recursion]\label{lem:recursion}
\[
  M_0(\mu,\sigma)=\Phi(\mu/\sigma), \qquad
  M_1(\mu,\sigma)=\mu\,\Phi(\mu/\sigma)+\sigma\,\varphi(\mu/\sigma),
\]
\[
  M_n(\mu,\sigma) = \mu\,M_{n-1}(\mu,\sigma) + (n-1)\sigma^2 M_{n-2}(\mu,\sigma), \qquad n\ge 2 .
\]
\end{lem}

\begin{proof}
Integration by parts on $z\varphi(z;\mu,\sigma^2)=\mu\varphi(z;\mu,\sigma^2)-\sigma^2\varphi'(z;\mu,\sigma^2)$.
\end{proof}

\begin{lem}[Score]\label{lem:score}
\[
  \frac{\partial M_n}{\partial \mu}(\mu,\sigma) = n\, M_{n-1}(\mu,\sigma), \qquad n\ge 1 .
\]
Consequently $\dfrac{\partial \log M_n}{\partial \mu} = n\,\dfrac{M_{n-1}}{M_n}$.
\end{lem}

\begin{proof}
Differentiate under the integral sign: $\partial_\mu\varphi(z;\mu,\sigma^2) =
(z-\mu)/\sigma^2\,\varphi(z;\mu,\sigma^2)$, so
$\partial_\mu M_n = (M_{n+1}-\mu M_n)/\sigma^2$; substitute the recursion
for $M_{n+1}$ to collapse this to $n M_{n-1}$.
\end{proof}

\section{Proofs}

\subsection{Proof of Lemma \ref{lem:kron}}
\begin{proof}[Proof of Lemma \ref{lem:kron}]
Throughout, write $\tau=\tau(t)$. By Assumption~\ref{ass:kron},
$B=\hat B\otimes I_p$, $C=e_1^\T\otimes I_p$, $q=e_k\otimes I_p$ and hence
$Q=qq^\T=(e_ke_k^\T)\otimes I_p$. We use the mixed-product identity
$(A\otimes B)(C\otimes D)=(AC)\otimes(BD)$ and
$e^{(\hat B\otimes I_p)s}=e^{\hat Bs}\otimes I_p$, the latter being immediate
from the power series since $(\hat B\otimes I_p)^n=\hat B^n\otimes I_p$.

\emph{(1)} We have
\[
  L(t) = Ce^{B\tau} = (e_1^\T\otimes I_p)(e^{\hat B\tau}\otimes I_p)
  = (e_1^\T e^{\hat B\tau})\otimes I_p = \hat c(\tau)^\T\otimes I_p .
\]

\emph{(2)} Since $L(t)^\T=\hat c(\tau)\otimes I_p$,
\[
  QL(t)^\T = \big(e_ke_k^\T\otimes I_p\big)\big(\hat c(\tau)\otimes I_p\big)
  = \big(e_k\,e_k^\T\hat c(\tau)\big)\otimes I_p
  = \hat c_k(\tau)\,(e_k\otimes I_p) = \hat c_k(\tau)\,q,
\]
as $e_k^\T\hat c(\tau)=\hat c_k(\tau)$ is a scalar.

\emph{(3)} Since $q^\T q=(e_k^\T e_k)\otimes I_p=I_p$, the matrix $Q=qq^\T$ is idempotent. Hence, by (2)
applied at time $s$,
\[
  L(s)QL(s)^\T = \big(QL(s)^\T\big)^\T\big(QL(s)^\T\big) = \hat c_k(\tau(s))^2\, q^\T q = \hat c_k(\tau(s))^2 I_p .
\]
By Lemma \ref{lem:condmeancov}, and substituting $r=T-s$,
\[ \Sigma(t)=\int_t^T \hat c_k(T-s)^2\dd s\; I_p=\int_0^{\tau(t)}\hat c_k(r)^2\dd r\; I_p=\sigma(t)^2I_p. \]
Moreover, $\hat c_k(s)=e_1^\T e^{\hat Bs}e_k=\sum_{j\ge0}e_1^\T\hat B^je_k\,s^j/j!$ is analytic, and by Assumption~\ref{ass:kron} at least one of its Taylor coefficients $e_1^\T\hat B^je_k$, $j\le k-1$, is nonzero. Hence $\hat c_k$ is not identically zero, so its zeros are isolated, and $\sigma(t)^2=\int_0^{\tau(t)}\hat c_k(r)^2\dd r>0$ for all $t<T$. For a chain of integrators, $e_i^\T\hat B=e_{i+1}^\T$ for $i<k$, so that $e_1^\T\hat B^je_k=0$ for $j<k-1$ and $e_1^\T\hat B^{k-1}e_k=1$, whatever the last row $\hat B_{k,:}$; then $\hat c_k(s)=s^{k-1}/(k-1)!+O(s^k)$ as $s\downarrow0$.

\emph{(4)} By Assumption \ref{ass:simplifying}, which holds by (3), $\mu(t,z)=u^\T L(t)z$, and
\[
  u^\T L(t) = (1\otimes u^\T)\big(\hat c(\tau)^\T\otimes I_p\big)
  = \hat c(\tau)^\T\otimes u^\T
  = \big(\hat c_1(\tau)u^\T,\ \dots,\ \hat c_k(\tau)u^\T\big).
\]
Applying this row vector to $z=(z^{(1)},\dots,z^{(k)})$ block-wise gives
\[ \mu(t,z)=\sum_{j=1}^k\hat c_j(\tau)\,u^\T z^{(j)}
=\sum_{j=1}^k\hat c_j(\tau)\langle u,z^{(j)}\rangle. \] Since $u^\T z^{(j)}=u^\T (uu^\T z^{(j)})$, $\mu(t,z)$ depends on $z$ only via $(I_k\otimes uu^\T)z$.

\emph{(5)} By the mixed-product identity,
\[
  (I_k\otimes P)(\hat B\otimes I_p) = \hat B\otimes P
  = (\hat B\otimes I_p)(I_k\otimes P)
\]
and
\[ (I_k\otimes P)(e_k\otimes I_p) = e_k\otimes P = (e_k\otimes I_p)P. \]
\end{proof}

\subsection{Proofs on limiting behaviour}
\begin{proof}[Proof of Theorem \ref{thm:hits-ray}]
Write $x_t:=\langle u_\perp, X_t^{\star\perp}\rangle$, $v_t:=\langle u_\perp, V_t^{\star\perp}\rangle$, $\tau=\tau(t)=T-t$,
and let $w_t$ denote the standard $1$-dimensional Brownian motion
$\langle u_\perp, W^{\star\perp}_t\rangle=\langle u_\perp, W_t^\star\rangle$, so that $X^{\star\perp}_t=x_tu_\perp$ and $\|X^{\star\perp}_t- \tfrac12\tau(W^{\star\perp}_t+C_0)\|=|x_t-\tfrac12\tau(w_t+c_0)|$ with $c_0:=\langle u_\perp,C_0\rangle$. By Corollary \ref{cor:ang radial processes},
\[
  \mathrm dx_t = v_t\,\mathrm dt, \qquad
  \mathrm dv_t = \mathrm dw_t - \frac{3}{\tau^2}x_t\,\mathrm dt - \frac3\tau v_t\,\mathrm dt.
\]

\emph{Step 1: decoupling.} Define $z_t=x_t/\tau(t)$. We have
\[ \dd z_t = \frac1{\tau^2} x_t \dd  t + \frac1{\tau} \dd x_t = \frac{s_t}{\tau} \,\dd t, \]
where $s_t := z_t+v_t$. Then
\[ \dd s_t = \dd z_t + \dd v_t= \frac{s_t}{\tau}\,\dd t+ \dd w_t -\frac{3}{\tau}s_t\, \dd t = -\frac{2}{\tau} s_t\, \dd t + \dd w_t, \]
so $s_t$ is a generalised ($\alpha=2$) Wiener bridge. Consequently
\[ \dd  \left(\frac{s_t}{\tau^2}\right) = \frac{1}{\tau^2} \,\dd s_t + \frac{2}{\tau^3} s_t\, \dd t = \frac{1}{\tau^2}\Big[-\frac2\tau s_t\,\dd t+\dd w_t\Big] + \frac2{\tau^3}s_t\,\dd t = \frac1{\tau^2} \,\dd w_t. \]

\emph{Step 2: rate for $s_t$ via the law of the iterated logarithm.}
Define the continuous martingale $M_t=\int_0^t \tau(s)^{-2} \,\dd w_s$, so that $s_t=\tau(t)^2\big[s_0\tau(0)^{-2}+M_t\big]$. Its quadratic variation is
\[ [M]_t = \int_0^t (T-s)^{-4} \,\dd s = \frac13\big(\tau(t)^{-3}-\tau(0)^{-3}\big) \sim \frac{\tau(t)^{-3}}{3} \quad\text{as } \tau(t)\to0. \]
By the Dubins--Schwarz theorem, $M_t=B_{[M]_t}$ for a standard Brownian motion $B$, and by Khinchin's law of the iterated logarithm,
\[ \limsup_{u\to\infty}\frac{|B_u|}{\sqrt{2u\log\log u}} = 1 \quad\text{a.s.} \]
Since $[M]_t\to\infty$ monotonically as $t\uparrow T$, this gives
\[ \limsup_{t\uparrow T}\frac{|M_t|}{\sqrt{2[M]_t\log\log[M]_t}} = 1 \quad\text{a.s.} \]
Now $\log[M]_t = 3\log(1/\tau(t)) + O(1)$ as $\tau(t)\to0$, so $\log\log[M]_t = \log\big(3\log(1/\tau(t))+O(1)\big) = \log\log(1/\tau(t)) + O(1)$, whence $\log\log[M]_t/\log\log(1/\tau(t))\to1$. Combined with $[M]_t\sim \tau(t)^{-3}/3$,
\[ \sqrt{2[M]_t\log\log[M]_t} \sim \sqrt{\tfrac23}\,\tau(t)^{-3/2}\sqrt{\log\log(1/\tau(t))} \quad\text{as }\tau(t)\to0, \]
and therefore
\[ \limsup_{t\uparrow T} \frac{\tau(t)^{3/2}|M_t|}{ \sqrt{\log\log(1/\tau(t))}} = \sqrt{\frac23} \quad\text{a.s.} \]
Since $\tau(t)^2 s_0\tau(0)^{-2}=O(\tau(t)^2)$ is of strictly smaller order than $\xi(t):=\sqrt{\tau(t)\log\log(1/\tau(t))}$, this yields
\begin{equation}\label{eq:limsup s_t}
	\limsup_{t\uparrow T} \frac{|s_t|}{\xi(t)} = \limsup_{t\uparrow T}\frac{\tau(t)^{3/2}|M_t|}{\sqrt{\log\log(1/\tau(t))}} = \sqrt{\frac23} \quad\text{a.s.}
\end{equation}

\emph{Step 3: the driftless coordinate $p_t$.} Define $p_t=v_t+3z_t$. Then
\[ \dd p_t = \dd v_t+3\,\dd z_t = \dd w_t -\frac{3}{\tau} s_t\, \dd t + \frac{3s_t}{\tau}\, \dd t = \dd w_t, \]
so $p_t=p_0+w_t$ with $p_0=v_0+3z_0=c_0$. From $p_t-s_t=2z_t=2x_t/\tau$,
\[ x_t - \frac\tau2(w_t+c_0) = -\frac\tau2 s_t. \]

\emph{Step 4: conclusion.} Take absolute values on both sides of the preceding display and use \eqref{eq:limsup s_t} to get
\[ \limsup_{t\uparrow T} \frac{\big|x_t-\tfrac\tau2(w_t+c_0)\big|}{\tau(t)\,\xi(t)} = \frac12\limsup_{t\uparrow T}\frac{|s_t|}{\xi(t)} = \frac12\sqrt{\frac23} = \frac{1}{\sqrt6} \quad\text{a.s.} \]
Since $\tau(t)\xi(t)\to0$ and $\tfrac\tau2(w_t+c_0)\to0$ (as $w_t+c_0\to w_T+c_0$, finite a.s.\ by continuity of $w$), it follows that $x_t\to0$ a.s.\ as $t\uparrow T$; by continuity of the sample paths of $Z_t^\star$ on $[0,T]$, $u_\perp^\top X_T^\star=0$, i.e.\ $X_T^\star\in\mathrm{span}(u)$.
\end{proof}

\subsection{Proofs for the angular and radial process}

\begin{proof}[Proof of Theorem \ref{thm:closed-system}, angular part]
We apply $I_k\otimes\Pi$ to the SDE for $Z^\star$ in Corollary \ref{cor:guided-block}, written as
\[ \dd Z_t^\star=BZ_t^\star\,\dd t+\hat c_k(\tau)\,q\,\alpha(t,Z_t^\star)\,\dd t+q\,\dd W_t^\star. \]
By
Lemma~\ref{lem:kron}(5), $(I_k\otimes\Pi)BZ_t^\star=B(I_k\otimes\Pi)Z_t^\star=BA^{\star}_t$. Again by Lemma~\ref{lem:kron}(5),
\[
  (I_k\otimes\Pi)\,\hat c_k(\tau)\,q\,\alpha(t,Z_t^\star)= \hat c_k(\tau)\,q\,\Pi\alpha(t,Z_t^\star)
  = -\frac{\hat c_k(\tau)}{\sigma(t)^2}\,q\,\Pi m(t,Z_t^\star),
\]
where the final equality follows from $\Pi u=0$ and $\Pi^2=\Pi$. Since
\[\Pi m(t,Z_t^\star)=\Pi L(t)Z_t^\star=(\hat c(\tau)^\top\otimes\Pi)Z_t^\star=(\hat c(\tau)^\top\otimes I_p)(I_k\otimes\Pi)Z_t^\star=\hat c(\tau)^\top A^{\star}_t\]
(using Lemma~\ref{lem:kron}(1)), this gives the stated singular term. Finally, by
Lemma~\ref{lem:kron}(5),
$(I_k\otimes\Pi)q\,\dd W_t^\star = q\,\Pi\,\dd W_t^\star=q\,\dd W^{\star\ang}_t$.
\end{proof}

\begin{proof}[Proof of Theorem \ref{thm:closed-system}, radial part]
Identical to the proof of the angular part, with $\Pi$
replaced throughout by $\Pi^{\rad}=uu^\top$.  The only step that differs is
\[
  \Pi^{\rad}\alpha(t,z)=\Pi^{\rad}\left[(p-1)\frac{M_{p-2}(\mu,\sigma)}{M_{p-1}(\mu,\sigma)}\,u \;-\; \frac{\|\Pi m\|}{\sigma^2}\,e_\perp\right]
  = (p-1)\frac{M_{p-2}(\mu,\sigma)}{M_{p-1}(\mu,\sigma)}\,u,
\]
since $\Pi^\rad u=u$ and $\Pi^\rad e_\perp=0$.
By Lemma \ref{lem:kron}(4), $\mu(t,Z_t^\star)=\mu(t,R^{\star}_t)$ is recoverable from $R^{\star}_t$ alone. The noise term becomes
$(I_k\otimes\Pi^{\rad})q\,\mathrm dW_t^\star=q\,\Pi^\rad\mathrm dW_t^\star=q\,\mathrm dW_t^{\star\rad}$, by Lemma~\ref{lem:kron}(5).
Independence of $W_t^{\star\rad}$ and $W^{\star\ang}_t$ follows since
they are jointly Gaussian with $\mathrm{Cov}(W^{\star\rad}_t,W^{\star\ang}_t)=\Pi^\rad\Pi^\T\, t=uu^\T(I_p-uu^\T)\,t=0$; they are projections of the same standard Brownian motion $W_t^\star$
onto the orthogonal subspaces $\mathrm{span}(u)$ and $u^\perp$.
\end{proof}

\subsection{Proof of Lemma \ref{lem:conditional-law}}
Let $t<T$, $A\in\scr F_t$ and $B\subseteq S^{p-1}$ Borel. By the Markov property,
$\mathbb{L}(U\in B\mid\scr F_t)=\int_Bg_u(t,Z_t)\,\varsigma(\dd u)$, and therefore, as $A\in\scr F_t$,
\begin{align*}
  \int_B\mathbb{L}_u(A)\,g_u(0,z_0)\,\varsigma(\dd u)&=\mathbb{L}(A\cap\{U\in B\})\\ &	
  = \mathbb{L}\big[\mathbb{L}(A\cap\{U\in B\}\mid \scr{F}_t)\big]
  = \mathbb{L}\big[\ind_A\,\mathbb{L}(U\in B\mid \scr{F}_t)\big]\\
  &= \mathbb{L}\Big[\ind_A\int_B g_u(t,Z_t)\,\varsigma(\dd u)\Big]
  =\int_B\mathbb{L}\big[\ind_A\,g_u(t,Z_t)\big]\,\varsigma(\dd u).
\end{align*}
Since $\scr F_t$ is countably generated, there is a $\varsigma$-null set $N$ such that for all $u\notin N$, all
rational $t<T$ and all $A\in\scr F_t$,
\[ \mathbb{L}_u(A)\,g_u(0,z_0)=\mathbb{L}\big[\ind_A\,g_u(t,Z_t)\big], \]
that is,
\begin{equation}\label{eq:PPlin_u PPlin}
  \frac{\dd\mathbb{L}_u}{\dd\mathbb{L}}\bigg|_{\scr F_t}(Z)=\frac{g_u(t,Z_t)}{g_u(0,z_0)} .
\end{equation}
This extends to all $t<T$, because $g_u(t,Z_t)$ is a $\mathbb{L}$-martingale on $[0,T)$: for $t<s<T$ with $s$
rational, the density on $\scr F_t$ is $\mathbb{L}[g_u(s,Z_s)\mid\scr F_t]/g_u(0,z_0)=g_u(t,Z_t)/g_u(0,z_0)$.
The martingale property holds for every $u$: writing $\ell_t(x\mid z)$ for the Gaussian density of $X_T$ given
$Z_t=z$, we have $g_u(t,z)=\int_0^\infty R^{p-1}\ell_t(Ru\mid z)\,\dd R$, and for every $x$ the process
$\ell_t(x\mid Z_t)$ is a $\mathbb{L}$-martingale by the Markov property, so the claim follows by Tonelli's theorem.

Under $\mathbb{L}$, $W_t=q^+\big(Z_t-z_0-\int_0^tBZ_s\,\dd s\big)$ is an $(\scr F_t)$-Brownian motion. Let $u$ be such that \eqref{eq:PPlin_u PPlin} holds and let
$L_t=g_u(t,Z_t)/g_u(0,z_0)$, $t<T$. Since $g_u$ is smooth on $[0,T)\times\RR^d$ by Theorem~\ref{thm:general case g and nabla log g}, $L$ is a continuous semimartingale
by It\^o's formula, and since it is a $\mathbb{L}$-martingale its finite-variation part vanishes. Hence
$\dd L_t=L_t\,\theta_t^\T\dd W_t$ with $\theta_t=q^\T\nabla_z\log g_u(t,Z_t)$, that is, $L_t=\mathcal{E}\left(\int_0^t \theta_s^\T \dd W_s\right)$, with $\mathcal{E}$ denoting the stochastic exponential. 
By Girsanov's theorem, applied on $[0,t]$ for each $t<T$, $W_t-\int_0^t\theta_s\,\dd s$ is an $(\scr F_t)$-Brownian motion on
$[0,T)$ under $\mathbb{L}_u$. Since $Q=qq^\T$ and $q^+q=I$, this process equals $W^\star_t$, and
\[
  Z_t=z_0+\int_0^t\big[BZ_s+Q\nabla_z\log g_u(s,Z_s)\big]\dd s+q\,W^\star_t ,
\]
so $Z$ solves \eqref{eq:Zcirc} with driving process $W^\star$.

\subsection{Local Lipschitz property of the guiding term}

The following regularity of the guiding term is used in the proof of Theorem \ref{lem:strong-nonlinear}.

\begin{lem}\label{lem:g-lipschitz}
Let $u\in S^{p-1}$ and $t<T$. For every $r>0$ there is a constant $K$ such that
\[
  \big\|\nabla_z\log g_u(s,z)-\nabla_z\log g_u(s,z')\big\|\le K\|z-z'\|
  \qquad\text{for all } s\in[0,t] \text{ and } \|z\|,\|z'\|\le r .
\]
That is, $z\mapsto\nabla_z\log g_u(s,z)$ is locally Lipschitz, uniformly in $s\in[0,t]$.
\end{lem}

\begin{proof}[Proof of Lemma \ref{lem:g-lipschitz}]
Since $m(s,z)=L(s)z$, Theorem~\ref{thm:general case g and nabla log g} gives
\[
  \nabla_z\log g_u(s,z)=(p-1)\,\xi_s^2\,\rho\big(\mu(s,z),\xi_s\big)\,L(s)^\T\Sigma(s)^{-1}u
  -L(s)^\T\Sigma(s)^{-1}\Pi_sL(s)\,z,
\]
with
\[  \rho(\mu,\xi):=\frac{M_{p-2}(\mu,\xi)}{M_{p-1}(\mu,\xi)} \qquad \text{and}\qquad 
 \mu(s,z)=\xi_s^2\,u^\T\Sigma(s)^{-1}L(s)z.\] The matrices $L(s)=Ce^{B(T-s)}$ and
$\Sigma(s)=\int_s^TL(r)QL(r)^\T\dd r$ are continuous in $s$, and $\Sigma(s)$ is positive definite for $s\le t<T$.
Hence $\Sigma(s)^{-1}$, $\xi_s=(u^\T\Sigma(s)^{-1}u)^{-1/2}$, $\Pi_s$ and the coefficients of the linear maps
$z\mapsto\mu(s,z)$ and $z\mapsto L(s)^\T\Sigma(s)^{-1}\Pi_sL(s)z$ are continuous on the compact interval $[0,t]$.
In particular they are bounded there, and $\xi_s\ge\xi_{\min}>0$. The second term above is therefore Lipschitz in $z$,
with a constant that does not depend on $s\in[0,t]$.

For the first term, $M_n(\mu,\xi)=\int_{-\mu/\xi}^\infty(\mu+\xi y)^n\phi(y)\dd y$ is positive and continuous on
$\RR\times(0,\infty)$. Moreover $\partial_\mu M_n=nM_{n-1}$ for $n\ge1$ by Lemma~\ref{lem:score}, and
$\partial_\mu M_0(\mu,\xi)=\xi^{-1}\phi(\mu/\xi)$. Hence $\rho$ is continuously differentiable in $\mu$, and
$\partial_\mu\rho$ is continuous on $\RR\times(0,\infty)$. The set
$\{(\mu(s,z),\xi_s)\colon s\in[0,t],\ \|z\|\le r\}$ is compact, as a continuous image of $[0,t]\times\{\|z\|\le r\}$, and
it is contained in $\RR\times(0,\infty)$. Denote this set by $K$. As $\partial_\mu\rho$ is continuous on $\RR\times(0,\infty)\supseteq K$ and $K$ is compact,
$c:=\max_{K}|\partial_\mu\rho|<\infty$. In particular, the continuous and positive function $M_{p-1}$ attains a positive minimum on $K$, so the denominators of $\rho$ and $\partial_\mu\rho$ are bounded away from zero there.

Now fix $s\in[0,t]$ and $z,z'$ with $\|z\|,\|z'\|\le r$. We show that
$|\rho(\mu(s,z),\xi_s)-\rho(\mu(s,z'),\xi_s)|\le c\,c'\|z-z'\|$, with $c$ and $c'$ not depending on $s$, in three steps.

\emph{Step 1: $\mu(s,\cdot)$ is linear with bounded norm.} We have $\mu(s,z)=a(s)^\T z$ with
$a(s)=\xi_s^2L(s)^\T\Sigma(s)^{-1}u$. Hence $\mu(s,z)-\mu(s,z')=a(s)^\T(z-z')$, and by the Cauchy--Schwarz inequality
\[ |\mu(s,z)-\mu(s,z')|\le\|a(s)\|\,\|z-z'\|\le c'\,\|z-z'\|,\qquad c':=\sup_{s\in[0,t]}\|a(s)\|. \]
Here $c'<\infty$, since $a$ is continuous on the compact interval $[0,t]$.

\emph{Step 2: mean value theorem in $\mu$.} Since $\xi_s$ does not depend on $z$, only the first argument of $\rho$
varies. The function $\mu\mapsto\rho(\mu,\xi_s)$ is continuously differentiable on $\RR$, so by the mean value theorem
there is a $\tilde\mu$ between $\mu(s,z')$ and $\mu(s,z)$ such that
\[ \rho(\mu(s,z),\xi_s)-\rho(\mu(s,z'),\xi_s)=\partial_\mu\rho(\tilde\mu,\xi_s)\,\big(\mu(s,z)-\mu(s,z')\big). \]

\emph{Step 3: the intermediate point lies in $K$.} The bound $|\partial_\mu\rho|\le c$ is only available on $K$, so
we need $(\tilde\mu,\xi_s)\in K$. Let $z_\lambda=z'+\lambda(z-z')$ for $\lambda\in[0,1]$. Since the ball
$\{\|z\|\le r\}$ is convex, $\|z_\lambda\|\le r$ for all $\lambda\in[0,1]$. By linearity of $\mu(s,\cdot)$,
\[ \mu(s,z_\lambda)=\mu(s,z')+\lambda\big(\mu(s,z)-\mu(s,z')\big),\qquad \lambda\in[0,1], \]
which takes every value between $\mu(s,z')$ and $\mu(s,z)$. Hence $\tilde\mu=\mu(s,z_{\tilde\lambda})$ for some
$\tilde\lambda\in[0,1]$, so that $(\tilde\mu,\xi_s)\in K$ and $|\partial_\mu\rho(\tilde\mu,\xi_s)|\le c$.

Combining the three steps,
\[ |\rho(\mu(s,z),\xi_s)-\rho(\mu(s,z'),\xi_s)|\le c\,|\mu(s,z)-\mu(s,z')|\le c\,c'\,\|z-z'\|, \]
for all $s\in[0,t]$ and $\|z\|,\|z'\|\le r$. Multiplying by the bounded coefficient $(p-1)\xi_s^2L(s)^\T\Sigma(s)^{-1}u$ shows that the
first term is Lipschitz on the ball as well, with a constant that does not depend on $s\in[0,t]$.
\end{proof}

\subsection{Proof of Corollary \ref{cor:linear-growth}}

We verify that boundedness of $f$ can be replaced by linear growth in the proofs of Theorem \ref{lem:strong-nonlinear} and Theorem \ref{prop:ac-nonlinear}. Boundedness is used in three places.

\emph{(i)} It is easy to see that \eqref{eq:for strong solution} applies under a linear growth condition as well, establishing part (1) of Theorem \ref{lem:strong-nonlinear}.

\emph{(ii)} In part (2) of that theorem, boundedness is used to verify that $(D_t)_{t<T}$ is a mean one martingale for which the limit $D_T$ exists, with $\mathbb{L}_uD_T=1$ and $D_T>0$. As $W^\star$ is an $L^2$-bounded martingale on $[0,T)$, it extends to a Brownian motion on $[0,T]$. By Example~3 in Section~6.2 of \cite{LiptserShiryaev2001}, $\mathbb{L}_uD_T=1$ holds as soon as $\sup_{t\le T}\mathbb{L}_u\exp\big(\delta\|f(t,Z_t)\|^2\big)<\infty$ for some $\delta>0$, and by linear growth of $f$ it suffices that $\sup_{t\le T}\mathbb{L}_u\exp\big(\delta\|Z_t\|^2\big)<\infty$. 
Under $\mathbb{L}$ the vector $(Z_t,X_T)$ is Gaussian. Hence, writing $X_T=Ru$, under $\mathbb{L}_u$ the law of $Z_t$ is a mixture over $R$ of Gaussian laws $N(\alpha_t+\Gamma_tRu,\Lambda_t)$, with $\alpha_t$, $\Gamma_t$ and $\Lambda_t$ bounded on $[0,T]$, and the mixing density of $R$, proportional to $R^{p-1}\phi\big(Ru;L(0)z_0,\Sigma(0)\big)$ on $(0,\infty)$, has Gaussian tails. Therefore $\sup_{t\le T}\mathbb{L}_u\exp\big(\delta\|Z_t\|^2\big)<\infty$ for $\delta>0$ small enough. Finally, $D_T>0$, as $\int_0^T\|f(s,Z_s)\|^2\dd s<\infty$ by continuity of the paths of $Z$ on $[0,T]$.

\emph{(iii)} The final part where boundedness of $f$ is used is in Step~1 of the proof of Theorem \ref{prop:ac-nonlinear}, together with the existence of $h_u$. If $f$ has linear growth, it is bounded on compact sets, and \eqref{eq:Z} has a unique strong solution by (i). Therefore, Theorem~1 in \cite{DelyonHu2006} remains applicable. Alternatively, $\mathbb{L}\Psi_T=1$ follows from case~(b) of Example~3 in Section~6.2 of \cite{LiptserShiryaev2001}, since under $\mathbb{L}$ the process $Z$ solves $\dd Z_t=BZ_t\dd t+q\dd W_t$, with a drift of linear growth and a constant diffusion coefficient. As $\Psi_T>0$, $\PP$ and $\mathbb{L}$ are equivalent, and the same argument, started at time $t$ in state $z$, gives the existence of $h_u(t,z)$ as before.

\section{Discretisation of the nearly constant speed model}\label{sec:discretisation_nsc_model}

Consider the planar model $\dd X_t=V_t\dd t$, $\dd V_t=b(V_t)\dd t+\dd W_t$ with
$b\colon\RR^2\to\RR^2$, which is \eqref{eq:Z} with $f=b$. Its guided process $Z^\circ=(X^\circ,V^\circ)$ is
defined by \eqref{eq:guided-nonlinear}, with the guiding function $g$ of the nearly-constant-velocity model. Recall the definition of the flow $\Phi_t(v)$ in Equation \eqref{eq:exactflow}.

\begin{lem}\label{lem:impl_split_ncv}
Let $M_t^{\circ\ang} = \Pi\big(X_t^\circ + \tau V_t^\circ\big)$. 
\begin{equation}\label{eq:ncv-sde-b}
\begin{split}
	\dd X_t^{\circ\ang} & = \frac{M_t^{\circ\ang} - X_t^{\circ\ang}}{\tau} \dd t, \qquad
	\dd M_t^{\circ\ang} = -3\frac{M_t^{\circ\ang}}{\tau} \dd t + \tau\,\Pi b(V^\circ_t)\dd t + \tau \dd W^{\circ\ang}_t, \\
	\dd X_t^{\circ\rad} & = V_t^{\circ\rad}\dd t, \qquad
	\dd V_t^{\circ\rad} = \Big[\Pi^\rad b(V^\circ_t)+\tau F\big(\langle u,X^{\circ\rad}_t+\tau V^{\circ\rad}_t\rangle,\sigma(t)\big)u\Big]\dd t + \dd W^{\circ\rad}_t .
\end{split}
\end{equation}
One step from $t$ to $t+\delta$ of the scheme that combines the exact flow of $b$ with the implicit treatment of the
term $-3M^{\circ\ang}_t/\tau$ is as follows. Let $\tau_\delta:=\tau(t+\delta)$ and
\[ \Delta=\Phi_\delta(V^\circ_t)-V^\circ_t+\delta\,\tau(t)F\big(\langle u,X^{\circ\rad}_t+\tau(t)V^{\circ\rad}_t\rangle,\sigma(t)\big)u+\big(W^\circ_{t+\delta}-W^\circ_t\big). \]
Then
\[
\begin{split}
 X^{\circ\rad}_{t+\delta}&=X^{\circ\rad}_t+\delta\,V^{\circ\rad}_t, \\   V^{\circ\rad}_{t+\delta}&=V^{\circ\rad}_t+\Pi^\rad\Delta,\\
 (\tau_\delta+3\delta)\,M_{t+\delta}^{\circ\ang} &= \tau_\delta\big(M_t^{\circ\ang} +\tau(t)\,\Pi\Delta\big),\\
 (\tau_\delta+\delta)\,X_{t+\delta}^{\circ\ang} &= \tau_\delta X_{t}^{\circ\ang} + \delta\, M_{t+\delta}^{\circ\ang},
\end{split}
\]
and $V^{\circ\ang}_{t+\delta}=(M^{\circ\ang}_{t+\delta}-X^{\circ\ang}_t)/(\tau_\delta+\delta)$, which equals $(M^{\circ\ang}_{t+\delta}-X^{\circ\ang}_{t+\delta})/\tau_\delta$ if $\tau_\delta>0$. 
\end{lem}
\begin{proof}
By \eqref{eq:guided-nonlinear}, $\dd V^\circ_t=\big[b(V^\circ_t)+\nabla_v\log g(t,Z^\circ_t)\big]\dd t+\dd W^\circ_t$, and by
Corollary~\ref{cor:ang radial processes} the guiding term is
$\nabla_v\log g=\tau F(\langle u,X+\tau V\rangle,\sigma)u-3\Pi(X+\tau V)/\tau^2$. Projecting with $\Pi^\rad$ gives the radial
equations. For the angular part, $\dd M^{\circ\ang}_t=\Pi\big(\dd X^\circ_t+\tau\dd V^\circ_t-V^\circ_t\dd t\big)=\tau\,\Pi\dd V^\circ_t$,
which gives the equation for $M^{\circ\ang}$, and the equation for $X^{\circ\ang}$ follows as in Lemma~\ref{lem:discretisation}.
In the scheme, $\Phi_\delta(V^\circ_t)-V^\circ_t$ replaces $\int_t^{t+\delta}b(V^\circ_s)\dd s$, and $\Delta$ collects all
increments of $V^\circ$ except the one due to $-3M^{\circ\ang}/\tau^2$, which has no radial component. The radial part is
an explicit Euler step. The angular equations are discretised implicitly in the term $-3M^{\circ\ang}/\tau$ and in the
drift of $X^{\circ\ang}$, as in Lemma~\ref{lem:impl_euler_ncv}, and explicitly in the other terms, which enter through
$\tau(t)\Pi\Delta$.
\end{proof}
Unlike in Corollary~\ref{cor:ang radial processes}, the radial and angular parts are coupled through $b$, which depends on the full velocity $V^\circ_t$; the guiding terms still separate. In the last step,
$\tau_\delta=0$, we have $M^{\circ\ang}_T=X^{\circ\ang}_T=0$, so that the discretised guided process ends exactly on the
ray; the formula for the velocity remains defined and gives $V^{\circ\ang}_T=-X^{\circ\ang}_t/\delta$, so that $X^{\circ\ang}_T=X^{\circ\ang}_t+\delta V^{\circ\ang}_T$ also holds in the last step.
For $b\equiv0$ the angular part reduces to the step of Lemma~\ref{lem:impl_euler_ncv}.

The weight \eqref{eq:weight-nonlinear} is accumulated along the same grid by the left-point rule,
\[ \int_t^{t+\delta}b(V^\circ_s)^\T\nabla_v\log g(s,Z^\circ_s)\dd s\approx\delta\, b(V^\circ_t)^\T\nabla_v\log g(t,Z^\circ_t). \]

\section{Bearings from a moving sensor}\label{sec:moving-sensor}

Throughout the paper the sensor is at the origin. Suppose instead that it moves along a known deterministic path
$s(t)\in\RR^p$, so that the bearing at time $T$ is $U=(X_T-s(T))/\|X_T-s(T)\|$ and $X_T$ is conditioned to lie on the
ray $\{s(T)+Ru\colon R>0\}$. Since $X_T-s(T)\mid Z_t=z\sim N\big(m(t,z)-s(T),\Sigma(t)\big)$, the guiding function and
its gradient are obtained from Theorem~\ref{thm:general case g and nabla log g} and Corollary~\ref{thm:gradlog_g} upon
replacing $m(t,z)$ by $m(t,z)-s(T)$; as $s(T)$ does not depend on $z$, $\Sigma(t)$, $\xi_t$, $\Pi_t$ and
$\nabla_z m(t,z)=L(t)$ are unchanged. The results of Section~\ref{sec:nonlinear} only involve $f$ and $g_u$ and therefore
hold without change. In the filter of Section~\ref{sec:smc}, the guided path on $[t_{n-1},t_n]$ is directed towards the
ray with apex $s(t_n)$, so that only the sensor positions at the observation times enter; for the exact first update the
prior mean of the position is shifted by $-s(t_1)$. If the position of the sensor at time $T$ is itself uncertain,
$S_T\sim N(\bar s,\Sigma_S)$ independent of $Z$, the same formulas apply with $s(T)$ replaced by $\bar s$ and $\Sigma(t)$ by
$\Sigma(t)+\Sigma_S$; as this covariance does not vanish when $t\uparrow T$, the conditioned process then no longer ends
on a ray.

A sensor manoeuvre is the classical way to make the range observable for a target with constant velocity
\citep{NardoneAidala1981}. Figure~\ref{fig:moving-sensor} is the analogue of Figure~\ref{fig:ncs-scenario} for a moving
sensor: the same true path ($\ka=400$, $\eta=0.01$) is observed from a sensor that starts at the origin, moves east at
the target speed $v_*=0.05$ for $t\le12$, across the initial line of sight of the target, and south afterwards. Compared
with Figure~\ref{fig:ncs-scenario}, the range is better determined during the first observations, as the motion of the
sensor across the line of sight gives parallax; later it is less precise, as the sensor ends far from the target and
moves nearly parallel to it, whereas in Figure~\ref{fig:ncs-scenario} the target passes the sensor at close range.

\begin{figure}[!ht]
\centering
\includegraphics[width=\textwidth]{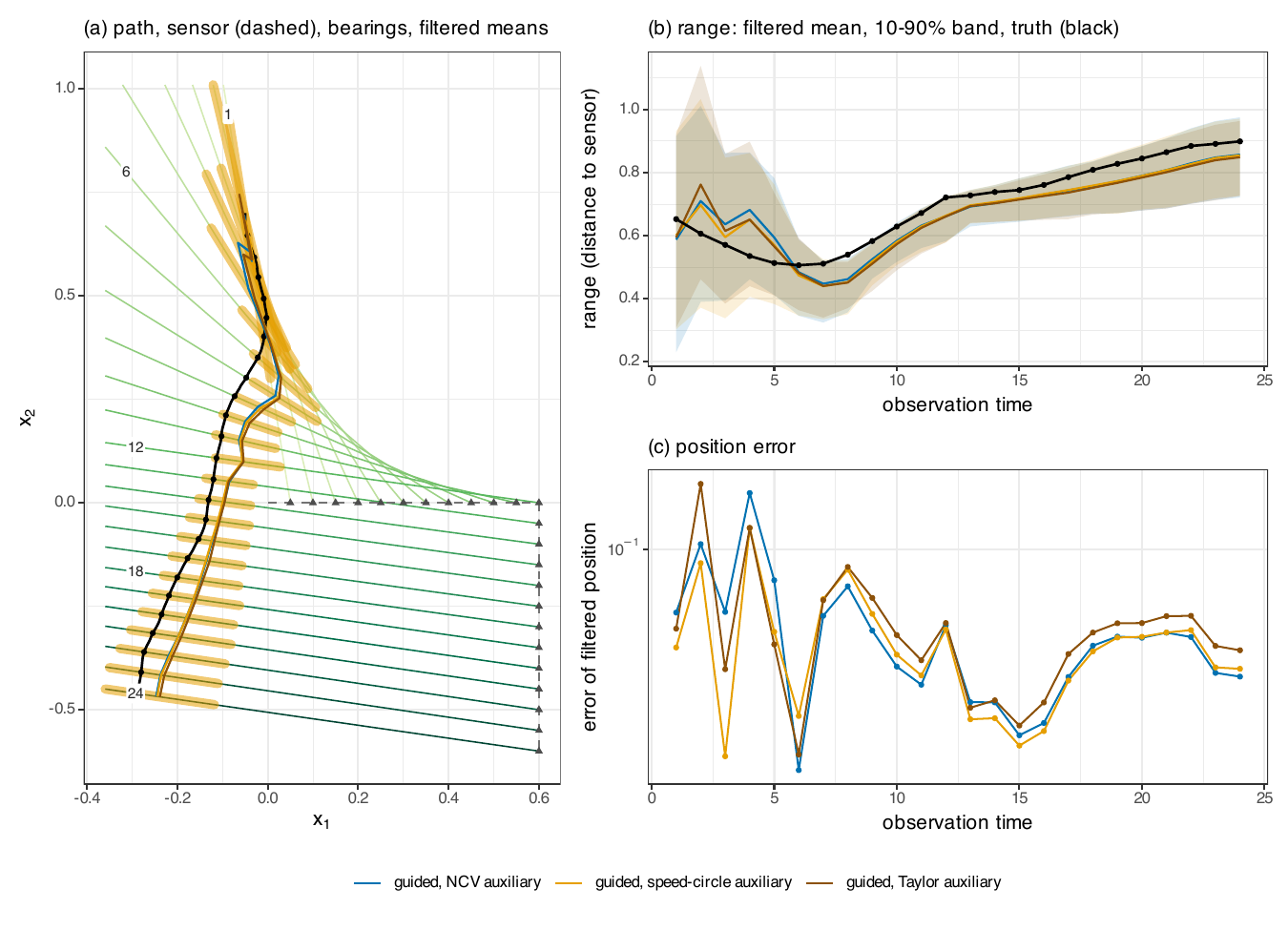}
\caption{As Figure~\ref{fig:ncs-scenario}, with the bearings observed from a moving sensor ($\ka=400$, $\eta=0.01$,
$N=5000$). (a) True path (black, dots at observation times), sensor path (dashed, triangles at the observation times),
observed bearings drawn from the sensor positions, filtered means, and the $10$--$90\%$ intervals of the filtering
distribution of the range for the speed-circle auxiliary. (b) Filtered mean of the range $\|X_{t_n}-s(t_n)\|$ (the
distance to the sensor) with the band between its $10\%$ and $90\%$ quantiles, and the true range (black). (c) Distance
between filtered mean and true position.}
\label{fig:moving-sensor}
\end{figure}
\clearpage   %

\acks Claude (Anthropic), including the model Claude Opus 5.5, was used to check derivations and proofs of the stated mathematical results for errors. The proof of Lemma~\ref{lem:g-lipschitz} was initially proposed by Claude and subsequently verified and revised by the author. The Julia code for the numerical experiments was written by Claude following the author's specification of the algorithms; Claude also wrote test scripts to check its correctness, and the author reviewed both the code and the tests. The choice and design of the numerical illustrations were made by the author. Claude was used for rewording small blocks of text of the manuscript. The author has checked all results and proofs and takes full responsibility for the content of the paper, including its mathematical correctness.
\competing There were no competing interests to declare which arose during the preparation or publication process of this article.

\bibliography{lit}

\begin{thebibliography}{10}

\bibitem{aidala1979kalman}
{\sc Aidala, V.~J.} (1979).
\newblock Kalman filter behavior in bearings-only tracking applications.
\newblock {\em IEEE Transactions on Aerospace and Electronic Systems\/} {\bf
  AES-15,} 29--39.

\bibitem{Baudoin2002}
{\sc Baudoin, F.} (2002).
\newblock Conditioned stochastic differential equations: theory, examples and
  application to finance.
\newblock {\em Stochastic Processes and their Applications\/} {\bf 100,}
  109--145.

\bibitem{bierkens2020}
{\sc Bierkens, J., van~der Meulen, F. and Schauer, M.} (2020).
\newblock Simulation of elliptic and hypo-elliptic conditional diffusions.
\newblock {\em Advances in Applied Probability\/} {\bf 52,} 173--212.

\bibitem{chopin2023}
{\sc Chopin, N., Fulop, A., Heng, J. and Thiery, A.~H.} (2023).
\newblock Computational doob's h-transforms for online filtering of discretely
  observed diffusions.
\newblock In {\em Proceedings of the 40th International Conference on Machine
  Learning}.
\newblock PMLR.
\newblock Honolulu, HI.
\newblock pp.~5904--5923.

\bibitem{Chopin2020introduction}
{\sc Chopin, N. and Papaspiliopoulos, O.} (2020).
\newblock {\em An Introduction to Sequential Monte Carlo}.
\newblock Springer series in statistics. Springer, Cham.

\bibitem{Clark2007SRF}
{\sc Clark, J. M.~C., Vinter, R.~B. and Yaqoob, M.~M.} (2007).
\newblock Shifted {R}ayleigh filter: A new algorithm for bearings-only
  tracking.
\newblock {\em IEEE Transactions on Aerospace and Electronic Systems\/} {\bf
  43,} 1373--1384.

\bibitem{DelyonHu2006}
{\sc Delyon, B. and Hu, Y.} (2006).
\newblock Simulation of conditioned diffusion and application to parameter
  estimation.
\newblock {\em Stochastic Processes and their Applications\/} {\bf 116,}
  1660--1675.

\bibitem{Farina1999}
{\sc Farina, A.} (1999).
\newblock Target tracking with bearings-only measurements.
\newblock {\em Signal Processing\/} {\bf 78,} 61--78.

\bibitem{Fearnhead2002}
{\sc Fearnhead, P.} (2002).
\newblock Markov chain monte carlo, sufficient statistics, and particle
  filters.
\newblock {\em Journal of Computational and Graphical Statistics\/} {\bf 11,}
  848--862.

\bibitem{FitzsimmonsPitmanYor1993}
{\sc Fitzsimmons, P., Pitman, J. and Yor, M.} (1993).
\newblock Markovian bridges: Construction, {P}alm interpretation, and splicing.
\newblock In {\em Seminar on Stochastic Processes, 1992}.
\newblock vol.~33 of {\em Progress in Probability}. Birkh{\"a}user, Boston
  pp.~101--134.

\bibitem{gordon1993novel}
{\sc Gordon, N.~J., Salmond, D.~J. and Smith, A. F.~M.} (1993).
\newblock Novel approach to nonlinear/non-{G}aussian {B}ayesian state
  estimation.
\newblock {\em IEE Proceedings F (Radar and Signal Processing)\/} {\bf 140,}
  107--113.

\bibitem{HermesLaSalle1969}
{\sc Hermes, H. and LaSalle, J.~P.} (1969).
\newblock {\em Functional Analysis and Time Optimal Control} vol.~56 of {\em
  Mathematics in Science and Engineering}.
\newblock Academic Press, New York.

\bibitem{HernandezStumpfhauser2017}
{\sc Hernandez-Stumpfhauser, D., Breidt, F.~J. and van~der Woerd, M.~J.}
  (2017).
\newblock The general projected normal distribution of arbitrary dimension:
  Modeling and {B}ayesian inference.
\newblock {\em Bayesian Analysis\/} {\bf 12,} 113--133.

\bibitem{Jacod1985}
{\sc Jacod, J.} (1985).
\newblock Grossissement initial, hypoth{\`e}se ({H}') et th{\'e}or{\`e}me de
  {G}irsanov.
\newblock In {\em Grossissements de filtrations: exemples et applications}. ed.
  T.~Jeulin and M.~Yor.
\newblock vol.~1118 of {\em Lecture Notes in Mathematics}. Springer, Berlin
  pp.~15--35.

\bibitem{KaratzasShreve1991}
{\sc Karatzas, I. and Shreve, S.~E.} (1991).
\newblock {\em {B}rownian Motion and Stochastic Calculus} second~ed. vol.~113
  of {\em Graduate Texts in Mathematics}.
\newblock Springer, New York.

\bibitem{Li2003survey1}
{\sc Li, X.~R. and Jilkov, V.~P.} (2003).
\newblock Survey of maneuvering target tracking. part i: Dynamic models.
\newblock {\em IEEE Transactions on Aerospace and Electronic Systems\/} {\bf
  39,} 1333--1364.

\bibitem{LiptserShiryaev2001}
{\sc Liptser, R.~S. and Shiryaev, A.~N.} (2001).
\newblock {\em Statistics of Random Processes {I}: General Theory} 2~ed.
\newblock Springer, Berlin.

\bibitem{Mallick2022}
{\sc Mallick, M., Tian, X., Zhu, Y. and Morelande, M.} (2022).
\newblock Angle-only filtering of a maneuvering target in {3D}.
\newblock {\em Sensors\/} {\bf 22,} 1422.

\bibitem{MansuyYor2006}
{\sc Mansuy, R. and Yor, M.} (2006).
\newblock {\em Random Times and Enlargements of Filtrations in a {B}rownian
  Setting} vol.~1873 of {\em Lecture Notes in Mathematics}.
\newblock Springer, Berlin.

\bibitem{NardoneAidala1981}
{\sc Nardone, S.~C. and Aidala, V.~J.} (1981).
\newblock Observability criteria for bearings-only target motion analysis.
\newblock {\em IEEE Transactions on Aerospace and Electronic Systems\/} {\bf
  AES-17,} 162--166.

\bibitem{Ozelci2013}
{\sc {\"O}zel{\c{c}}i, A.~C.} (2013).
\newblock Tracking and estimation algorithms for bearings only measurements.
\newblock {\em PhD thesis}.
\newblock Imperial College London.

\bibitem{PittShephard1999}
{\sc Pitt, M.~K. and Shephard, N.} (1999).
\newblock Filtering via simulation: Auxiliary particle filters.
\newblock {\em Journal of the American Statistical Association\/} {\bf 94,}
  590--599.

\bibitem{reiss_sde_lecture}
{\sc Rei{\ss}, M.}
\newblock Stochastic differential equations: Lecture notes 2007.
\newblock Humboldt-Universit{\"a}t zu Berlin.

\bibitem{robert1999monte}
{\sc Robert, C.~P. and Casella, G.} (1999).
\newblock {\em Monte Carlo statistical methods}.
\newblock Springer, New York.

\bibitem{RogersWilliams_vol2}
{\sc Rogers, L. C.~G. and Williams, D.} (2000).
\newblock {\em Diffusions, {M}arkov processes, and martingales. {V}ol. 2}.
\newblock Cambridge Mathematical Library. Cambridge University Press,
  Cambridge.
\newblock It\^{o} calculus, Reprint of the second (1994) edition.

\bibitem{Sarkka2013}
{\sc S\"arkk\"a, S.} (2013).
\newblock {\em Bayesian Filtering and Smoothing}.
\newblock Institute of Mathematical Statistics Textbooks. Cambridge University
  Press, Cambridge.

\bibitem{schauer2017guided}
{\sc Schauer, M., {van der Meulen}, F. and {van Zanten}, H.} (2017).
\newblock Guided proposals for simulating multi-dimensional diffusion bridges.
\newblock {\em Bernoulli\/} {\bf 23,} 2917--2950.

\bibitem{bffg}
{\sc van~der Meulen, F.~H. and Sommer, S.} (2025).
\newblock Backward filtering forward guiding.
\newblock {\em Journal of Machine Learning Research\/} {\bf 26,} 1--51.

\bibitem{WangGelfand2013}
{\sc Wang, F. and Gelfand, A.~E.} (2013).
\newblock Directional data analysis under the general projected normal
  distribution.
\newblock {\em Statistical Methodology\/} {\bf 10,} 113--127.

\bibitem{wang2009target}
{\sc Wang, X., Morelande, M. and Moran, B.} (2009).
\newblock Target motion analysis using single sensor bearings-only
  measurements.
\newblock In {\em 2009 2nd International Congress on Image and Signal
  Processing}.
\newblock IEEE.
\newblock Tianjin, China.
\newblock pp.~1--6.

\end{thebibliography}
\bibliographystyle{APT}

\end{document}